\documentclass[12pt,a4paper]{article}
\usepackage{amsmath,bm}
\usepackage{amssymb}
\usepackage{mathrsfs}
\usepackage{graphicx}
\usepackage{amsfonts}
\usepackage{enumerate}
\usepackage{thm}
\usepackage{color}
\usepackage[pdfstartview=FitH]{hyperref}
\makeatletter
\def\tank#1{\protected@xdef\@thanks{\@thanks
 \protect\footnotetext[0]{#1}}}
\def\bigfoot{

 \@footnotetext}
\makeatother

\newcommand{\ea}{\end{array}}

\allowdisplaybreaks

\newtheorem{theorem}{Theorem}[section]
\newtheorem{lem}{Lemma}[section]
\newtheorem{prp}[theorem]{Proposition}
\newtheorem{thm}[theorem]{Theorem}
\newtheorem{cor}[theorem]{Corollary}
\newtheorem{dfn}[theorem]{Definition}
\newtheorem{exmp}[theorem]{Example}
\newtheorem{remark}{Remark}[section]

\newtheorem{Hypothesis}[theorem]{Hypothesis}

\def\beq{\begin{equation}}
\def\nneq{\end{equation}}

\def\bthm{\begin{thm}}
\def\nthm{\end{thm}}

\def\R{{\bf R}}
\def\Rd{{\bf R}^d}
\def\a{{\alpha}}
\def\b{{\beta}}
\def\d{{\mathrm{d}}}

\def\K{{\bf K}}
\def\e{{\varepsilon}}

\def\l{{\lambda}}
\def\<{{\langle}}
\def\>{{\rangle}}

\def\e{\varepsilon}

{\theorembodyfont{\rmfamily}
}

\title{Strong existence and uniqueness of solutions of SDEs with time dependent Kato class coefficients}

\footnotesize{\author{
 Saisai Yang$^{1,}$\thanks{yangss@mail.ustc.edu.cn},\ \
 Tusheng Zhang$^{2,}$\thanks{Tusheng.Zhang@manchester.ac.uk}\\
 {\em $^1$ School of Mathematical Sciences,}\\
 {\em University of Science and Technology of China,}\\
 {\em Hefei, 230026, China}\\
 {\em $^2$ Department of Mathematics, University of Manchester,}\\
 {\em Oxford Road, Manchester, M13 9PL, UK}\\
}
\date{}
\newenvironment{proof}{\par\noindent{\bf Proof:}}{\hspace*{\fill}$\blacksquare$\par}
\begin{document}
\maketitle
\noindent \textbf{Abstract:} Consider stochastic differential equations (SDEs) in $\Rd$: $dX_t=dW_t+b(t,X_t)\d t$, where $W$ is a Brownian motion, $b(\cdot, \cdot)$ is a measurable vector field. It is known that if $|b|^2(\cdot, \cdot)=|b|^2(\cdot)$ belongs to the Kato class $\K_{d,2}$, then there is a weak solution to the SDE.
In this article we show that if $|b|^2$ belongs to the Kato class $\K_{d,\a}$ for some $\a \in (0,2)$ ($\a$ can be arbitrarily close to $2$), then there exists a unique strong solution  to the  stochastic differential equations, extending the results in the existing literature as demonstrated by examples.  Furthermore, we allow the drift to be time-dependent. The new regularity estimates we established  for the solutions of  parabolic equations with Kato class coefficients  play a crucial role.


\vspace{4mm}


\vspace{3mm}

\numberwithin{equation}{section}
\tableofcontents
\section{Introduction}
\indent
In this article,
we are concerned with the strong solutions to the following stochastic differential equations (SDEs) in $\Rd$ ($d \ge 2$):
\begin{equation}\label{1.1}
X_t=x+W_t+\int_0^tb(s,X_s)\d s,\ \forall t>0,
\end{equation}
where $W:=(W_t)_{t\ge 0}$ is a $d$-dimensional standard Brownian motion on a given probability space $(\Omega, \mathcal{F}, (\mathcal{F})_{t\ge 0}, P)$ satisfying the usual assumptions, and $b$ is a measurable  $\Rd$-valued function on $(0,\infty)\times \Rd$.
We stress that the assumption that the diffusion matrix is identity is just for the clarity of the exposition. Our method also works for general, non-degenerate diffusion matrices.

\medskip

The strong solutions of SDEs with singular drift have been investigated by many authors. In the celebrated work \cite{Zvonkin},  Zvonkin  introduced a quasi-isometric transformation of the phase space that can convert a stochastic differential equation with a non-zero singular drift into a SDE without drift. This method is now called Zvonkin transformation. From then on, there are many papers devoted to extending the  Zvonkin transformation in various ways to obtain the strong solutions of stochastic differential equations with  singular coefficients.  Krylov and R\"{o}ckner in \cite{Krylov} showed that there exist unique strong solutions to Brownian motion with drifts $b$ that are in the class $L_p$-$L_q:= \bigcap\limits_{T> 0}L^{p }((0,T);L^{q}(\Rd))$ for some $p,q \ge 2$ and $\frac{2}{p }+ \frac{d}{q}<1$. Then Xicheng Zhang in
\cite{ZhangXC4}  extended the above work to  the case that the diffusion coefficients are uniformly non-degenerate and belong to some Sobolev spaces.
Recently, in the critical case when the drift $b$ belongs to $L_p$-$L_q$ with $\frac{2}{p }+ \frac{d}{q}=1$, Beck et al. in \cite{Beck} obtained the existence and uniqueness of strong solutions  for almost every starting point $x$. Xia et al. in \cite{ZhangXC5} obtained the strong well-posedness of SDE (\ref{1.1}) for $b \in \tilde L_{p}^{q}$ (see (2.2) in \cite{ZhangXC5} for the precise definition of the space $\tilde L_{p}^{q}$) with $p,q \ge 2$. We also mention  \cite{Gyongy}, \cite{Veretennikov}, \cite{ZhangXC5}, \cite{Xie} and \cite{ZhangXC3} for further related works.
\vskip 0.3cm
On the other hand,  the existence and uniqueness of weak solutions to SDEs require  much less conditions on the coefficients. We refer readers to \cite{BassChen}, \cite{Jin},  \cite{KimSong} and  \cite{Kinzebulatov} for references.
\vskip 0.3cm

 If $|b|^2$ belongs to the Kato class $\K_{d,2}$, then using Kahamiskii's inequality and Girsanov transformation, it is known that there exists a weak solution to SDE (\ref{1.1}). It is natural to ask the following question:
\vskip 0.3cm
Is there a unique strong solution to SDE (\ref{1.1}) when $|b|^2$ belongs to the Kato class $\K_{d,\alpha}$ for some $\alpha\in (0,2)$ arbitrarily close to $2$?
\vskip 0.3cm
The purpose of this article is to give a positive answer to the above question. We also allow  the drifts to be time-dependent. It is easy to verify that the  $L_p$-$L_q$ (or $\tilde L_{p}^{q}$) condition appeared in the literature mentioned above satisfy our assumption. Moreover, as the examples below show,  there are plenty of functions satisfying our conditions do not fulfill the $L_p$-$L_q$ (or $\tilde L_{p}^{q}$) assumptions.
\vskip 0.3cm
  To prove the uniqueness of the strong solutions to SDEs (\ref{1.1}), we apply the Zvonkin transformation. To this end we need to establish the required regularity of the solutions of the associated parabolic partial differential equations. Because of the singularity of the drift $b$, the classical Krylov's estimates for semi-martingales (see \cite[Lemma 5.1]{Krylov2}) do not apply here. Moreover, since $b$ may not be in the class $L_p$-$L_q$, the Calderon-Zygmund inequality for parabolic equations  can not be used  anymore. Therefore, one of the main tasks  is to establish  new regularity estimates for the associated parabolic equations from scratch under the  Kato class conditions.

\medskip

The rest of this article is arranged as follows. In Section 2, we state the main result and give some examples. In Section 3, we give several equivalent conditions for the Kato class functions. And we also obtain some useful properties for the maximal function of kato class functions. In Section 4, We consider  the Kolmogorov equations associated with SDEs (\ref{1.1}). We derived the existence and uniqueness of the mild solutions to the Kolmogorov equations and prove that the second derivative of the mild solutions belong to the desired Kato class. In Section 5, we  prove  the pathwise uniqueness of solutions of equation  (\ref{1.1}).  To use the Zvonkin transformation, we establish the Krylov type estimates for the SDEs (\ref{1.1}) and exploit  the regularity estimates for  the maximal function of the second derivative of the solutions of the Kolmogorov equation.

\medskip


\medskip
\section{Statement of the main result}
In this section, we state the main result and give some examples. We start with  the definition of  the Kato classes of functions:
\begin{dfn} \label{D1.1}
Given $0<\alpha\leq 2$, a measurable function $f:(0,\infty)\times \Rd \to \R$ is said to be in the Kato class $\K_{d,\alpha}$, if for each $\lambda>0$,
\begin{eqnarray}\label{1.2}
&&\lim_{T \to 0}\sup_{t \ge 0, x\in \R^d}\int_0^{T} \int_{\Rd} s^{-\frac{d+2-\a}{2}}e^{-\frac{\lambda |x-y|^2}{2s}}|f(t+s,y)|\d y \d s=0, \nonumber \\
&&\lim_{T \to 0}\sup_{t \ge 0, x\in \R^d}\int_0^{T} \int_{\Rd} s^{-\frac{d+2-\a}{2}}e^{-\frac{\lambda |x-y|^2}{2s}}|f(t-s,y)|\d y \d s=0, \nonumber
\end{eqnarray}
where we set $f(t,x):=0$ for $(t,x) \in (-\infty, 0) \times \Rd$.
\end{dfn}

The main result of the paper reads as follows.

\begin{thm}  \label{T1.6}
Assume for any $T>0$, $|b(t,x)|^2 I_{(0,T)}(t) \in \K_{d,\a}$ for some $\a \in (0,2)$. Then for each $x \in \Rd$, there exists a unique strong solution to SDE (\ref{1.1}).
\end{thm}

\begin{remark}
By the Proposition \ref{P2.3} below,  we see that for any $T>0$, \\
$|b(t,x)|^2 I_{(0,T)}(t) \in \K_{d,\a}$ for some $\a \in (0,2)$ is equivalent to one of the following two conditions:

$(\mathrm{i})$ For any $T>0$, there exists $\b \in (0,2)$ such that
\begin{eqnarray}
\lim_{\tau \to 0}\sup_{t \in [0,T], x\in \R^d}\int_0^{\tau} \int_{\Rd} s^{-\frac{d+\b}{2}}e^{-\frac{ |x-y|^2}{2s}}|b(t+s,y)|^2 \d y \d s=0. \nonumber
\end{eqnarray}

$(\mathrm{ii})$ For any $T>0$, there exist constants $p>d$ and $M>0$ depending on $T$ such that
\begin{eqnarray} \label{1.3}
\sup_{t \in [0,T], x\in \Rd}\int_{0}^{r^2} \int_{B(x,r)}|b (t+s,y)|^2\d y \d s \le M r^p,  \ \forall 0<r <1,
\end{eqnarray}
where $B(x,r)$ is the ball in $\Rd$ centered at $x$ with radius $r$.
\end{remark}

\vskip 0.5cm
It is easy to see that all bounded measurable functions satisfy (\ref{1.3}) with $p=d+2$, and that the $ L^{p_1}$-$L^{p_2}$ class functions with $p_1,p_2 \ge 2$ and $\frac{2}{p_1}+ \frac{d}{p_2}<1$ (Krylov-R\"ockner condition in \cite{Krylov}) satisfy (\ref{1.3}) with $p=d+2(1-\frac{2}{p_1}-\frac{d}{p_2})$. Now we give some examples of functions satisfying (\ref{1.3}), which are not even contained in the critical case of $L_p$-$L_q$ conditions.


\begin{exmp}
Given $\a_i > -\frac{1}{2}$ for each $1 \le i \le d$. Set $f(x_1,x_2,\cdots,x_d):=\prod\limits_{1 \le i \le d}(|x_i| \wedge 1)^{\a_i}$. If $\sum\limits_{1\le i \le d} \a_i >-1$, then $f$ satisfies (\ref{1.3}) with $p=d+2+2\sum\limits_{1\le i \le d}\a_i$. However, if $\a_i\le -\frac{1}{q}$ for some $1\le i \le d$ and $q>2$, then $f \notin L^{q}_{loc}(\Rd)$.
\end{exmp}

\begin{exmp}
Given $\a \in (0,\frac{1}{d})$. Set $r_0=0$ and $r_i:=i^{-\frac{1}{d \a}}$ for each $1 \le i \le d$. Let $x_n:=(2\sum\limits_{1\le i \le n-1}r_i+r_n,0,\cdots,0)\in \Rd$. Define $f(x):=\sum\limits_{n\ge 1}|x-x_n|^{\a-1} I_{B(x_n, r_n)}(x)$. Then  one can verify that $f$ satisfies (\ref{1.3}) with $p=d+2\a$. However, $f \notin L^{d}_{loc}(\Rd)$.
\end{exmp}

\begin{exmp}
Let $\a \in (-\frac{1}{2},-\frac{1}{2d}]$ and $g_i$ be a bounded measurable function on $\R$ for each $2 \le i \le d$. Define
\begin{eqnarray}
f(t,x_1,x_2,\cdots,x_d):= t^{-\frac{1}{4}}  (|x_1| \wedge 1)^{\a}  \prod\limits_{2 \le i \le d} g_i(x_i). \nonumber
\end{eqnarray}
Then $f$ satisfies (\ref{1.3}) with $p=d+2\a+1$. However, $f \notin L^{p_1}((0,1);L^{p_2}(B(0,1)))$ for any $p_1$, $p_2$ with  $\frac{2}{p_1}+ \frac{d}{p_2} \le 1$.
\end{exmp}

%
%
%
%
%
%

\section{Time dependent Kato class functions}

In this section, we will provide equivalent/sufficient conditions for functions to be in the Kato Class and obtain  some important properties for  the maximal functions. Throughout, for a measurable function $f$ defined on $[0,T] \times \Rd$, we will let $f$ vanish outside of $[0,T] \times \Rd$ when there is no danger of confusion. The letter $c$ with or without subscripts stand for an unimportant positive constant, whose value may be different  in different places.

\vskip 0.5cm
First of all,  using Kolomogrov-Chapman equation for the transition function of a Brownian motion  it is easy to see that the following Lemma holds (see \cite[Proposition 1]{Jin}):
\begin{lem} \label{L2.2}
Assume $f \in \K_{d,2-\a}$ for some $\a \in [0,2)$. Then for any $\lambda,T>0$,
\begin{eqnarray} \label{2.8} \nonumber
\sup_{t\ge 0, x\in \R^d}\int_0^{T} \int_{\Rd} s^{-\frac{d+\a}{2}}e^{-\frac{\lambda|x-y|^2}{2s}}|f(t+s,y)|\d y \d s<\infty.
\end{eqnarray}
\end{lem}

\vskip 0.4cm

The following result is contained in \cite[Lemma 2.3]{Wang}.

\begin{lem}  \label{L2.2-1}
Given a measure $\nu$ on $\Rd$ and positive constants $\delta $ and $r$. Then for any positive integer $N$ with $\frac{3}{N}\le r$, we have
\begin{eqnarray} \nonumber
&& \sup_{x\in \R^d}\int_{ |x-y|\geq r }e^{-\delta |x-y|^2} \nu(dy)\le \sup_{z \in \R^d} \nu (B(z,\frac{1}{N})) \sum_{k\in Z^d: |k|\geq 2Nr-2}e^{-\delta \frac{1}{4}|\frac{k}{2N}|^2},  \\
&& \sup_{x\in \R^d}\int_{\R^d}e^{-\delta |x-y|^2} \nu (dy)\leq \sup\limits_{z \in \R^d} \nu (B(z,1)) e^{\delta }(\sum_{m=-\infty}^{\infty} e^{-\frac{\delta }{8}m^2})^d. \nonumber
\end{eqnarray}
\end{lem}

\vskip 0.4cm

Given a measurable function $f:(0,\infty)\times \Rd \to \R$, we introduce the following hypothesis:
\begin{Hypothesis} [$\mathrm{H_p}$] \label{H2.1}
There exist constants $p>d$ and $M_1>0$ such that
\begin{eqnarray} \label{2.2} \nonumber
\sup_{t \ge 0, x\in \Rd}\int_{0}^{r^2} \int_{B(x,r)}|f(t+s,y)| \d y \d s \le M_1r^p,  \ \forall 0<r <1.
\end{eqnarray}
\end{Hypothesis}
\vskip 0.5cm
\begin{lem}  \label{L2.3}
Assume $f$ satisfies the Hypothesis  $\mathrm{H_p}$. Then there exists $M_2>0$ such that for any $0< r < 2$,
\begin{eqnarray}
&& \sup_{t\ge 0, x\in \Rd}\int_{0}^{T} \int_{B(x, r)} |f(t+s,y)|\d y \d s \le M_2 T^{\frac{p-d}{2}} r^{d},  \ \forall T\le  r^2, \label{2.27}  \\
&& \sup_{t\ge 0, x\in \Rd}\int_{0}^{T} \int_{B(x, r)} |f(t+s,y)|\d y \d s \le M_2 T r^{p-2}, \ \forall T>  r^2 . \label{2.28}
\end{eqnarray}
\end{lem}
\begin{proof}
Let $r \in (0,2)$. If $T\le r^2$, then $T \in (2^{-2n}r^2, 2^{-2n+2}r^2]$ for some $n \ge 1$.
Thus,
\begin{eqnarray}  \nonumber
\begin{split}
\int_{0}^{ T} \int_{B(x, r)} |f(t+s,y)|\d y \d s  & \le   \int_{0}^{ 2^{-2n+2}r^2} \int_{B(x, r)} |f(t+s,y)|\d y \d s \\
& \le 4 \sup_{t\ge 0} \int_{0}^{ 2^{-2n}r^2} \int_{B(x, r)} |f(t+s,y)|\d y \d s  \\
& \le c  2^{nd} \sup_{t\ge 0, x \in \Rd} \int_{0}^{ 2^{-2n}r^2} \int_{B(x, 2^{-n}r)} |f(t+s,y)|\d y \d s  \\
& \le c  2^{nd}2^{-np} r^p \le c T^{\frac{p-d}{2}} r^{d}.
\end{split}
\end{eqnarray}

If $T> r^2$, then $T \in (n r^2,  (n+1)r^2]$ for some $n \ge 1$. Thus,
\begin{eqnarray}  \nonumber
\begin{split}
\int_{0}^{ T} \int_{B(x, r)} |f(t+s,y)|\d y \d s & \le \int_{0}^{ (n+1) r^2} \int_{B(x, r)} |f(t+s,y)|\d y \d s  \\
& \le c (n+1)  \sup_{t\ge 0, x \in \Rd} \int_{0}^{\frac{r^2}{4} } \int_{B(x,  \frac{r}{2})} |f(t+s,y)|\d y \d s  \\
& \le c (n+1)  r^p \le c T r^{p-2}.
\end{split}
\end{eqnarray}
\end{proof}

\vskip 0.4cm

Now we show that the Hypothesis  $\mathrm{H_p}$ is an equivalent condition for the Kato class.

\begin{prp} \label{P2.3}
$(\mathrm{i})$ Assume there exists a constant $p>d$ such that
\begin{eqnarray} \label{2.3-1}
\sup_{t \ge 0, x\in \Rd}  \int_{0}^{1} \int_{\Rd}s^{-\frac{p}{2}}e^{-\frac{|x-y|^2}{2s}} |f(t+s,y)|\d y \d s <\infty.
\end{eqnarray}
Then $f$ satisfies Hypothesis  $\mathrm{H_p}$.

$(\mathrm{ii})$ If $f$ satisfies Hypothesis  $\mathrm{H_p}$,
then $f \in \K_{d,2-\a}$ for any $\a \in [0, p-d)$.
\end{prp}

\begin{proof}
First we show (i). Assume (\ref{2.3-1}) holds, then for $0<r<1$, $t\ge 0$ and $x\in \Rd$,
\begin{eqnarray} \nonumber
\begin{split}
\int_{\frac{r^2}{2}}^{ r^2 } \int_{\Rd}s^{-\frac{p}{2}}e^{-\frac{|x-y|^2}{2s}} |f(t+s,y)|\d y \d s & \ge  \int_{\frac{r^2}{2}}^{ r^2 } \int_{|x-y|<r} s^{-\frac{p}{2}}e^{-\frac{|x-y|^2}{2s}} |f(t+s,y)|\d y \d s \\
& \ge  \int_{\frac{r^2}{2}}^{ r^2 } \int_{|x-y|<r} s^{-\frac{p}{2}}e^{-\frac{r^2}{2s}} |f(t+s,y)|\d y \d s \\
& \ge c   r^{- p} \int_{\frac{r^2}{2}}^{ r^2 } \int_{|x-y|<r} |f(t+s,y)|\d y \d s.
\end{split}
\end{eqnarray}
Therefore for any $0<r<1$,
\begin{eqnarray} \nonumber
\begin{split}
\sup_{t \ge 0, x \in \Rd} \int_{\frac{r^2}{2}}^{ r^2 } \int_{|x-y|<r} |f(t+s,y)|\d y \d s \le  c   r^{p},
\end{split}
\end{eqnarray}
which yields that
\begin{eqnarray} \nonumber
\begin{split}
 \int_0^{r^2} \int_{|x-y|<r} |f(t+s,y)|\d y \d s & \le  \sum_{n \ge 0} \int_{2^{-n-1}r^2}^{ 2^{-n} r^2 } \int_{|x-y|<r}   |f(t+s,y)|\d y \d s \\
& \le  c \sum_{n \ge 0} 2^{\frac{nd}{2}} \sup_{ x \in \Rd}  \int_{2^{-n-1}r^2}^{ 2^{-n} r^2 } \int_{|x-y|<2^{-\frac{n}{2}} r}   |f(t+s,y)|\d y \d s \\
& \le  c \sum_{n \ge 0}   2^{\frac{n(d-p)}{2}} r ^{p} \le c r ^{p}. \\
\end{split}
\end{eqnarray}
Hence $f$ satisfies the Hypothesis  $\mathrm{H_p}$.

\vskip 0.3cm

Now we show (ii). Assume $\mathrm{H_p}$ holds, and fix a positive constant $\a<p-d$. Then the following limits hold.
\begin{eqnarray}
&&\lim_{r \to 0}\sup_{t\ge 0, x\in \Rd}\int_{0}^{ r} \int_{B(x, 1)} |f(t+s,y)|\d y \d s =0, \label{2.3}\\
&&\lim_{r \to 0}\sup_{t\ge 0, x\in \Rd}\int_{0}^{ r^2} \int_{B(x,r) \setminus B(x, \sqrt {s})}|x-y|^{-d-\a} |f(t+s,y)|\d y \d s =0, \label{2.4}\\
&&\lim_{r \to 0}\sup_{t\ge 0, x\in \Rd}\int_{0}^{ r^2} s^{-\frac{d+\a}{2}}\int_{B(x, \sqrt {s})} |f(t+s,y)|\d y \d s =0. \label{2.5}
\end{eqnarray}

(\ref{2.3}) is a direct consequence of (\ref{2.27}).

Let us prove (\ref{2.4}). Take $\a < \b_1 <p-d$ and $d+\a-p+2 < \b_2<2$. Then we have
\begin{eqnarray}  \nonumber
\begin{split}
& \ \ \ \ \int_{0}^{ r^2} \int_{B(x,r) \setminus B(x, \sqrt {s})}|x-y|^{-d-\a} |f(t+s,y)|\d y \d s \\
&=\sum_{k \ge 0} \sum_{n \ge 0}  \int_{ 2^{-k-1}r^2}^{ 2^{-k}r^2}   \int_{(B(x,2^{-\frac{n}{2}}r) \setminus B(x, 2^{-\frac{n+1}{2}}r)) \setminus B(x, \sqrt {s})}   |x-y|^{ -d-\a} |f(t+s,y)|\d y \d s \\
& \le  \sum_{k \ge 0} \sum_{0 \le n \le k+1}  \int_{ 2^{-k-1}r^2}^{ 2^{-k}r^2} s^{-\frac{\b_1}{2}}  \int_{(B(x,2^{-\frac{n}{2}}r) \setminus B(x, 2^{-\frac{n+1}{2}}r)) \setminus B(x, \sqrt {s})}   |x-y|^{\b_1 -d-\a} |f(t+s,y)|\d y \d s \\
& \ \ \ \ + \sum_{k \ge 0} \sum_{ n > k+1}  \int_{ 2^{-k-1}r^2}^{ 2^{-k}r^2}  s^{-\frac{\b_2}{2}} \int_{(B(x,2^{-\frac{n}{2}}r) \setminus B(x, 2^{-\frac{n+1}{2}}r)) \setminus B(x, \sqrt {s})}   |x-y|^{\b_2 -d-\a} |f(t+s,y)|\d y \d s \\
& \le  \sum_{k \ge 0} \sum_{0 \le n \le k+1}  \int_{ 2^{-k-1}r^2}^{ 2^{-k}r^2} s^{-\frac{\b_1}{2}}  \int_{ B(x,2^{-\frac{n}{2}}r) \setminus B(x, 2^{-\frac{n+1}{2}}r)  }   |x-y|^{\b_1 -d-\a} |f(t+s,y)|\d y \d s \\
& \ \ \ \ + \sum_{k \ge 0} \sum_{ n > k+1}  \int_{ 2^{-k-1}r^2}^{ 2^{-k}r^2}  s^{-\frac{\b_2}{2}} \int_{ B(x,2^{-\frac{n}{2}}r) \setminus B(x, 2^{-\frac{n+1}{2}}r) }   |x-y|^{\b_2 -d-\a} |f(t+s,y)|\d y \d s . \\
\end{split}
\end{eqnarray}
Note that for any $\b \in \R$, $|x-y|^{\b}  \le c 2^{-\frac{n \b}{2}}r^{\b}$ if $y \in B(x,2^{-\frac{n}{2}}r) \setminus B(x, 2^{-\frac{n+1}{2}}r)$.
Thus by Lemma \ref{L2.3},
\begin{eqnarray}  \nonumber
\begin{split}
& \ \ \ \ \int_{0}^{ r^2} \int_{B(x,r) \setminus B(x, \sqrt {s})}|x-y|^{-d-\a} |f(t+s,y)|\d y \d s \\
& \le  c \sum_{k \ge 0} \sum_{0 \le n \le k+1}  2^{ \frac{\b_1(k+1)}{2}} r^{-\b_1}  2^{-\frac{n(\b_1 -d-\a)}{2}} r^{\b_1 -d-\a} \int_{ 2^{-k-1}r^2}^{ 2^{-k}r^2}   \int_{B(x,2^{-\frac{n}{2}}r)  }   |f(t+s,y)|\d y \d s \\
& \ \ \ \ +c \sum_{k \ge 0} \sum_{ n > k+1}  2^{ \frac{\b_2(k+1)}{2}} r^{-\b_2}  2^{-\frac{ n(\b_2 -d-\a)}{2}} r^{\b_2 -d-\a}  \int_{ 2^{-k-1}r^2}^{ 2^{-k}r^2}    \int_{ B(x,2^{-\frac{n}{2}}r) }     |f(t+s,y)|\d y \d s \\
& \le  c \sum_{k \ge 0} \sum_{0 \le n \le k+1}  2^{ \frac{\b_1(k+1)}{2}} r^{-\b_1}  2^{-\frac{n(\b_1 -d-\a)}{2}} r^{\b_1 -d-\a}  2^{-\frac{(k+1)(p-d)}{2}}r^{p-d}   2^{-\frac{nd}{2}}r^{d}  \\
& \ \ \ \ + c  \sum_{k \ge 0} \sum_{ n > k+1}  2^{ \frac{\b_2(k+1)}{2}} r^{-\b_2}  2^{-\frac{n(\b_2 -d-\a)}{2}} r^{\b_2 -d-\a}    2^{-k-1}r^2    2^{-\frac{n(p-2)}{2}}r^{p-2}  \\
& \le  c r^{ p -d-\a} \sum_{k \ge 0} 2^{ \frac{(\b_1-p+d)(k+1)}{2}}   \sum_{0 \le n \le k+1}    2^{-\frac{n (\b_1  -\a)}{2}}        \\
& \ \ \ \ + c r^{ p -d-\a} \sum_{k \ge 0}  2^{ \frac{(\b_2-2)(k+1)}{2}} \sum_{ n > k+1}    2^{-\frac{n(\b_2 -d-\a+p-2)}{2}} \le  c r^{ p -d-\a}.
\end{split}
\end{eqnarray}
Hence (\ref{2.4}) follows.

(\ref{2.5}) follows from the following bound:
\begin{eqnarray}  \nonumber
\begin{split}
& \ \ \ \ \int_{0}^{ r^2} s^{-\frac{d+\a}{2}}\int_{B(x, \sqrt {s})} |f(t+s,y)|\d y \d s  \\
& =  \sum_{n \ge 0} \int_{ 2^{-n-1}r^2}^{ 2^{-n}r^2} s^{-\frac{d+\a}{2}} \int_{B(x, \sqrt {s}) } |f(t+s,y)|\d y \d s \\
& \le  \sum_{n \ge 0} 2^{\frac{(n+1)(d+\a)}{2} }r^{- d-\a }\int_{ 2^{-n-1}r^2}^{ 2^{-n}r^2} \int_{B(x, 2^{-\frac{n}{2}}r) } |f(t+s,y)|\d y \d s \\
& \le M_2  \sum_{n \ge 0} 2^{\frac{(n+1)(d+\a)}{2} }r^{- d-\a } 2^{-\frac{(n+1)(p-d)}{2} } r^{ p-d }  2^{-\frac{nd}{2}}r^d \\
&=cr^{p-d-\a}\sum_{n \ge 0} 2^{\frac{n(d+\a-p)}{2}}.
\end{split}
\end{eqnarray}


Finally we show that $f \in \K_{d,2-\a}$. By Lemma \ref{L2.2-1} we have that for any $t,T,r,\delta>0$
\begin{eqnarray}  \nonumber
\begin{split}
&\ \ \ \ \int_0^T\int_{|x-y|\ge r}e^{-\delta |x-y|^2} |f(t+s,y)|  \d y\d s \\
&\le \sup_{z \in \Rd} \int_0^T\int_{|z-y|\le \frac{1}{N}} |f(t+s,y)|  \d y\d s \left(\sum_{-\infty<n<\infty}e^{-\frac{\delta n^2}{16N^2}}\right)^d,
\end{split}
\end{eqnarray}
where $N$ is some positive integer larger than $ \frac{3}{r}$. Therefore together with (\ref{2.3}) we get
\begin{eqnarray} \label{2.7}
\begin{split}
&\ \ \ \ \lim_{T \to 0} \sup_{t\ge 0,x\in \Rd}\int_0^T\int_{|x-y|\ge r}e^{-\delta |x-y|^2} |f(t+s,y)|  \d y\d s=0.
\end{split}
\end{eqnarray}
 Taking $0<T\le r^2<1$ and $\l>0$, then we have
\begin{eqnarray} \nonumber
\begin{split}
& \ \ \ \ \sup_{t\ge 0,x\in \Rd}\int_0^{ T} \int_{\Rd}s^{-\frac{d+\a}{2}}e^{-\frac {\l |x-y|^2}{2s}} |f(t+s,y)|\d y \d s \\
& \le \sup_{t\ge 0,x\in \Rd}\int_0^{r^2} \int_{|x-y|<\sqrt s}s^{-\frac{d+\a}{2}}e^{-\frac {\l |x-y|^2}{2s}} |f(t+s,y)|\d y \d s \\
&\ \ \ \ + \sup_{t\ge 0,x\in \Rd}\int_0^{r^2}  \int_{ \sqrt s \le |x-y|<r }s^{-\frac{d+\a}{2}}e^{-\frac {\l |x-y|^2}{2s}} |f(t+s,y)|\d y \d s\\
&\ \ \ \ +\sup_{t\ge 0,x\in \Rd}\int_0^{ T} \int_{|x-y|\ge r}s^{-\frac{d+\a}{2}}e^{-\frac{\lambda r^2}{4s}} e^{-\frac {\l |x-y|^2}{4}} |f(t+s,y)|\d y \d s \\
& \le \sup_{t\ge 0,x\in \Rd}\int_0^{r^2} \int_{|x-y|<\sqrt s}s^{-\frac{d+\a}{2}} |f(t+s,y)|\d y \d s \\
&\ \ \ \ +c \sup_{t\ge 0,x\in \Rd}\int_0^{r^2}  \int_{ \sqrt s \le |x-y|<r }|x-y|^{-d-\a}  |f(t+s,y)|\d y \d s\\
&\ \ \ \ +c r^{-d-\a}\sup_{t\ge 0,x\in \Rd}\int_0^{ T} \int_{|x-y|\ge r}  e^{-\frac {\l |x-y|^2}{4}} |f(t+s,y)|\d y \d s.
\end{split}
\end{eqnarray}
First letting $T \to 0$ and then $r \to 0$, by (\ref{2.4})-(\ref{2.7}), it follows that
\begin{eqnarray}  \label{2.8-1} \nonumber
\begin{split}
&\lim_{T \to 0}\sup_{t\ge 0,x\in \Rd}\int_0^{ T} \int_{\Rd}s^{-\frac{d+\a}{2}}e^{-\frac {\l |x-y|^2}{2s}} |f(t+s,y)|\d y \d s =0,
\end{split}
\end{eqnarray}
By a similar argument we also have
\begin{eqnarray}   \label{2.9-1} \nonumber
\begin{split}
&\lim_{T \to 0}\sup_{t\ge 0,x\in \Rd}\int_0^{ T} \int_{\Rd}s^{-\frac{d+\a}{2}}e^{-\frac {\l |x-y|^2}{2s}} |f(t-s,y)|\d y \d s =0,
\end{split}
\end{eqnarray}
We thus have proved that  $f \in \K_{d,2-\a}$.
\end{proof}

\vskip 0.5cm

Combining with Proposition \ref{P2.3}, we obtain the following result.

\begin{cor} \label{C2.3}
A measurable function $f$ satisfies Hypothesis  $\mathrm{H_p}$ if and only if $f \in \K_{d, \a}$ for some $\a \in (0, 2)$.
\end{cor}

\vskip 0.3cm

Next result shows that if $|f|^2$ belongs to some Kato class, $|f|$ is also in some Kato class with a different index.

\begin{lem}  \label{L2.4}
Given a measurable function $f$. If $|f|^2 \in \K_{d,\a}$ for some $\a \in (0,2)$, then $f \in \K_{d,1-\b}$ for any $\b \in [0,\frac{2-\a}{2})$.
\end{lem}
\begin{proof} Let $\b \in [0,\frac{2-\a}{2})$. By H\"{o}lder's inequality we have
\begin{eqnarray}  \nonumber
\begin{split}
& \ \ \ \  \int_0^{T} \int_{\Rd} s^{-\frac{d+1+\b}{2}}e^{-\frac{|x-y|^2}{2s}}|f(t \pm s,y)|\d y \d s \\
& \le ( \int_0^{T} \int_{\Rd} s^{-\frac{d+2-\a}{2}}e^{-\frac{|x-y|^2}{2s}}|f(t \pm s,y)|^2\d y \d s)^{\frac{1}{2}}( \int_0^{T} \int_{\Rd} s^{-\frac{d+2\b+\a}{2}}e^{-\frac{|x-y|^2}{2s}}\d y \d s)^{\frac{1}{2}} \\
& \le c( \int_0^{T} \int_{\Rd} s^{-\frac{d+2-\a}{2}}e^{-\frac{|x-y|^2}{2s}}|f(t \pm s,y)|^2\d y \d s)^{\frac{1}{2}}.
\end{split}
\end{eqnarray}
This yields that $f \in \K_{d,1-\b}$ according to the condition on $|f|^2$.
\end{proof}

\vspace{4mm}

Next we will prove  some properties of the local Hardy-Littlewood maximal functions which will be used later.
For a measurable function $f:(0,\infty)\times \Rd \to \R^n$ and $R>0$, let $M_R f(t,x)$ denote the local Hardy-Littlewood maximal function of $f$ given by
\begin{equation} \nonumber
M_R f(t,x):=\sup_{0<\delta \le R}\frac{1}{m(B(x,\delta))}  \int_{B(x,\delta)}  |f( t,y)|\d y,
\end{equation}
where $m(B(x,\delta))$ stands for the volume of the ball $B(x,\delta)$. $M_{\infty}f$ is the maximal function of $f$ defined in \cite{Stein}. Here is the result.

\begin{lem}  \label{L2.5}
Assume $| f |^2 $ satisfies the Hypothesis  $\mathrm{H_p}$. Then for any $R>0$, $|M_R f|^2$ satisfies the Hypothesis  $\mathrm{H_q}$ for any $q \in (d,p)$.
\end{lem}
\begin{proof}
Fix $(t,x_0) \in [0,\infty) \times \Rd$, $R>0$ and $0<r<1$. Set $f_1(s,x):=f(s,x)  I_{|f|\ge \frac{\a}{2}}  I_{|x-x_0|< 2r}$ and $f_2(s,x):=f(s,x)  I_{|f|\ge \frac{\a}{2}}  I_{|x-x_0|\ge 2r}$. Then
\begin{eqnarray}  \nonumber
\begin{split}
\{M_Rf > \a\} \subset \{ M_Rf_1 > \frac{\a}{4} \} \cup \{ M_Rf_2 > \frac{\a}{4} \}.
\end{split}
\end{eqnarray}
Therefore,
\begin{eqnarray}  \label{2.40}
\begin{split}
& \ \ \ \ \int_{0}^{ r^2} \int_{B(x_0,r) }   M_R f(t+s,y) ^2 \d y \d s \\
&= 2\int_{0}^{ r^2} \int_{B(x_0,r) }  \int_0^{ M_R f(t+s,y) } \a \d \a \d y \d s \\
&= 2 \int_{0}^{ r^2} \d s \int_0^{\infty} \a \d \a \int_{B(x_0,r) \cap \{  M_R f(t+s,y) > \a \}}   \d y  \\
&\le  2\int_{0}^{ r^2} \d s \int_0^{\infty} \a \d \a \int_{   M_Rf_1 > \frac{\a}{4}  }   \d y + 2 \int_{0}^{ r^2} \d s \int_0^{\infty} \a \d \a \int_{B(x_0,r) \cap \{ M_Rf_2 > \frac{\a}{4}  \}}   \d y \\
&\le 2\int_{0}^{ r^2} \d s \int_{ \Rd}   \d y \int_0^{4M_Rf_1}\a \d \a +  2 \int_{0}^{ r^2} \d s \int_0^{\infty} \a \d \a \int_{B(x_0,r) \cap \{ M_Rf_2 > \frac{\a}{4}  \}}   \d y\\
&\le  16 \int_{0}^{ r^2} \d s \int_{\Rd} |M_Rf_1 |^2   \d y+ 2 \int_{0}^{ r^2} \d s \int_0^{\infty} \a \d \a \int_{B(x_0,r) \cap \{ M_Rf_2 > \frac{\a}{4}  \}}   \d y:=\mathrm{I}_1+\mathrm{I}_2.
\end{split}
\end{eqnarray}

For the term $\mathrm{I}_1$, by \cite[Theorem 1 in section 1]{Stein}, we have
\begin{eqnarray}  \nonumber
\begin{split}
 \int_{\Rd} |M_Rf_1 |^2   \d y &\le \int_{\Rd} |M_{\infty}f_1 |^2   \d y\\
&\le c  \int_{\Rd} |f_1|^2   \d y \le c   \int_{B(x_0,2r)} |f|^2   \d y.
\end{split}
\end{eqnarray}
Hence by Lemma \ref{L2.3},
\begin{eqnarray}  \label{2.41}
\begin{split}
\mathrm{I}_1 & \le c \int_{0}^{ r^2} \d s   \int_{B(x_0,2r)} |f|^2   \d y \le c r^p.
\end{split}
\end{eqnarray}

Now we turn to the term $ \mathrm{I}_2$. For $y \in B(x_0,r)$, if $\delta< r$, by the definition of $f_2$, we see that $\int_{B(y,\delta)}  |f_2(s,z)|\d z =0$. If $r \le \delta \le R$, then $B(y,\delta) \subset B(x_0,r+\delta)$. Thus, for $y \in B(x_0,r)$,
\begin{eqnarray}  \nonumber
\begin{split}
M_R f_2(t+s, y)&= \sup_{r \le \delta \le R} \frac{1}{m(B(y,\delta))}\int_{B(y,\delta)}   | f_2(t+s,z)| \d z  \\
&\le \sup_{r \le \delta \le R} \frac{2\a^{-1}}{m(B(y,\delta))}  \int_{B(x_0,r+\delta) \setminus B(x_0, r)}  | f(t+s,z)|^2 \d z  \\
&\le   \sup_{r \le \delta \le R} \frac{2\a^{-1}}{m(B(y,\delta))}  (r+\delta)^d \int_{B(x_0,r+\delta) \setminus B(x_0, r)} |z-x_0|^{-d} | f (t+s,z)|^2 \d z \\
& \le c \a^{-1}  \sup_{r \le \delta \le R}  (\frac{r}{\delta}+1)^d \int_{B(x_0,r+R) \setminus B(x_0, r)} |z-x_0|^{-d} | f (t+s,z)|^2 \d z\\
&\le  c \a^{-1}  \int_{B(x_0,r+R) \setminus B(x_0, r)} |z-x_0|^{-d} | f (t+s,z)|^2 \d z := c \a^{-1}  m_{t+s}.
\end{split}
\end{eqnarray}
It follows that if $\a > 2\sqrt {cm_{t+s}}$, then $B(x_0,r) \cap \{ M_Rf_2 > \frac{\a}{4}  \}=\emptyset$. Hence we have
\begin{eqnarray}  \label{2.42}
\begin{split}
 \mathrm{I}_2 & =  \int_{0}^{ r^2} \d s \int_0^{\infty} \a \d \a \int_{B(x_0,r) \cap \{ M_Rf_2 > \frac{\a}{4}  \}}   \d y \\
&\le  \int_{0}^{ r^2} \d s \int_0^{2 \sqrt {cm_{t+s}}} \a \d \a \int_{B(x_0,r) }   \d y \\
&\le c r^d \int_{0}^{ r^2}   m_{t+s}  \d s.
\end{split}
\end{eqnarray}

On the other hand,  for any $0<\b < p-d$, by Lemma \ref{L2.3} we have
\begin{eqnarray}  \label{2.42-1}
\begin{split}
& \ \ \ \ \int_0^{r^2} m_{t+s} \d s=\int_{0}^{ r^2} \int_{B(x_0,r+R) \setminus B(x_0, r)}|x-y|^{-d} |f(t+s,y)|^2 \d y \d s \\
&\le r^{-\b} \sum_{n \ge 0}  \int_{ 0}^{  r^2}   \int_{(B(x_0,2^{-n}(r+R)) \setminus B(x_0, 2^{-(n+1)}(r+R))) \setminus B(x_0, r)}   |x-y|^{ \b-d } |f(t+s,y)|^2\d y \d s \\
&\le r^{-\b}  \sum_{n \ge 0} 2^{(n+1)(d-\b)} (r+R)^{\b-d} \int_{ 0}^{  r^2}   \int_{ B(x_0,2^{-n}(r+R))  \setminus B(x_0, r)}    |f(t+s,y)|^2 \d y \d s \\
&\le c r^{-\b}  \sum_{n \ge 0} 2^{(n+1)(d-\b)} (r+R)^{\b-d}  r^{p-d}   2^{-nd}(r+R)^d\le c (r+R)^{\b }  r^{p-d-\b}.
\end{split}
\end{eqnarray}
Putting together (\ref{2.42}) and (\ref{2.42-1}) we obtain
\begin{equation}  \label{2.42-2}
 \mathrm{I}_2 \le c  r^{p-\b}.
\end{equation}
Combining (\ref{2.40}), (\ref{2.41}) and (\ref{2.42-2}), we find that $|M_Rf|^2$ satisfies the Hypothesis  $\mathrm{H_q}$ for any $q \in (d,p)$.
\end{proof}

\vspace{4mm}

Following the same argument as in the proof of \cite[Lemma 3.5]{ZhangXC2}, we have the following Lemma:

\begin{lem}  \label{L2.6}
Let $f \in W^{0,1}_{1,loc}((0,T)\times \Rd)$. Then there exists a $\d t \d x$-null set $A \subset (0,T)\times \Rd$ such that for any $R>0$ and $(t,x),(t,y) \in (0,T)\times \Rd \setminus A$ with $|x-y|\le R$,
\begin{eqnarray}  \nonumber
\begin{split}
|f(t,x)-f(t,y)| \le 2^d |x-y|(M_R|\nabla_x f|(t,x)+M_R|\nabla_x f|(t,y)).
\end{split}
\end{eqnarray}
\end{lem}

\section{Kolmogorov equations associated with SDEs}

In this section, we study the Kolmogorov equations associated with  SDEs (\ref{1.1}). We will provide the regularity results for the solutions which will be used in next section. Given a vector valued function $f:(0,T)\times \Rd \to \Rd$, consider the following backward second order parabolic equation:
\begin{align} \label{2.13}
\left\{
\begin{aligned}
&\partial_t u (t,x)+\frac{1}{2}\Delta u(t,x)+\< b(t,x),  \nabla_x u(t,x) \>   =f(t,x),&&\forall (t,x) \in (0,T)\times \Rd, \\
& u(T,x) = 0, && \forall x \in \Rd.
\end{aligned}
\right.
\end{align}

The following is the definition of the mild solution to  equation (\ref{2.13}):
\begin{dfn} \label{D2.1}
We say that $u \in C_b^{0,1}([0,T]) \times \Rd)$ is a mild solution to the equation (\ref{2.13}), if for each $(t,x) \in [0,T] \times \Rd$.
\begin{eqnarray} \nonumber
\begin{split}
u(t,x)=\int_t^T\int_{\Rd} q(s-t,x,y) [\< b(s,y),  \nabla_y u(s,y) \>-f(s,y)] \d y \d s,
\end{split}
\end{eqnarray}
where $ q(t,x,y)$ is the transition density function of the Brownian motion.
\end{dfn}

Fix a non-negative function $\varphi \in C_0^{\infty}((0,1) \times B(0,1))$ with $\int_{\R^{d+1}} \varphi(t,x)\d x \d t=1$. For any positive integer $n$, we put $\varphi_n(t, x):=2^{nd+n}\varphi(2^{n}t, 2^{n}x)$.



For $\lambda, T>0$ and $h \in \K_{d,1}$, define
\begin{eqnarray}   \nonumber
\begin{split}
\mathcal{N}^{\lambda}_h(T):=\sup_{t \ge 0, x\in \R^d}\int_t^{t+T} \int_{\Rd} (s-t)^{-\frac{d+1}{2}} e^{-\frac{\lambda |x-y|^2}{4(s-t)}} |h(s,y)|\d y \d s.
\end{split}
\end{eqnarray}
We need the following Lemma.

\begin{lem} \label{L3.2}
Given $h\in \K_{d,1}$ and $T>0$. Let $\{f_n\}_{n\ge 1}$ be a sequence of Borel measurable functions satisfying $\sup_{n\ge 1} \|f_n\|_{ L^{\infty} }<\infty$ and
\begin{eqnarray}  \label{3.2-1}
 \lim_{n \to \infty} \sup_{s,  |z| \in [0,  2^{-n})}  \sup_{(t,x)\in [0, T-s] \times \Rd} |f_n(t+s,x+z)-f_n(t,x)|=0.
\end{eqnarray}
Then for any $\lambda, R>0$,
\begin{eqnarray}
&& \sup_{n \ge 1}\mathcal{N}^{\lambda}_{h_n} (T) \le \mathcal{N}^{\lambda}_h (T), \label{2.19}   \\
&& \lim_{n \to \infty} \sup_{0\le t  \le T, x\in B(0,R)} | \int_t^{T}\int_{\Rd}  \nabla_x  q(s-t,x,y) f_n(s,y) ( h_n (s,y)-h (s,y) ) \d y \d s|=0, \label{2.20}
\end{eqnarray}
where $h_n(t,x):=\int_{\R^{d+1}} \varphi_n(t-s,x-y)h(s, y)\d y \d s$.
\end{lem}

\begin{proof} Using the Fubini's theorem, we have
\begin{eqnarray}  \label{3.4-1} \nonumber
\begin{split}
 &\ \ \ \ \int_t^{t+T} \int_{\Rd} (s-t)^{-\frac{d+1}{2}} e^{-\frac{\lambda |x-y|^2}{4(s-t)}} |h_n(s,y)|\d y \d s \\
&\le  \int_t^{t+T} \int_{\Rd} (s-t)^{-\frac{d+1}{2}} e^{-\frac{\lambda |x-y|^2}{4(s-t)}} \int_{\R^{d+1}} \varphi_n( \tau,z)|h(s-\tau, y-z)|\d z \d \tau \d y \d s \\
&= \int_{\R^{d+1}} \varphi_n( \tau,z) \d z \d \tau \int_t^{t+T} \int_{\Rd} (s-t)^{-\frac{d+1}{2}} e^{-\frac{\lambda |x-y|^2}{4(s-t)}}  |h(s-\tau, y-z)| \d y \d s \\
&= \int_{\R^{d+1}} \varphi_n( \tau,z) \d z \d \tau \int_{t-\tau}^{t-\tau+T} \int_{\Rd} (s+\tau-t)^{-\frac{d+1}{2}} e^{-\frac{\lambda |x-y-z|^2}{4(s+\tau-t)}} |h(s, y)| \d y \d s \\
&\le \mathcal{N}^{\lambda}_h(T),
\end{split}
\end{eqnarray}
which is (\ref{2.19}).

\vspace{3mm}

To show (\ref{2.20}), we first prove that for any $\delta \in (0,T )$,
\begin{eqnarray}  \label{3.6-1}
\begin{split}
\lim_{n \to \infty} \sup_{t \in [0,T-\delta], x\in B(0,R)} \int_{T-2^{-n}}^{T} \int_{\Rd}|\nabla_x q(s-t,x,y)| |h(s, y)| \d y \d s=0.
\end{split}
\end{eqnarray}
For $r>0$ and sufficiently large  integer $n$, we have
\begin{eqnarray}  \label{3.7-1}
\begin{split}
& \ \ \ \ \sup_{t \in [0,T-\delta], x\in B(0,R)} \int_{T-2^{-n}}^{T} \int_{\Rd}|\nabla_x q(s-t,x,y)| |h(s, y)| \d y \d s \\
&\le \sup_{t \in [0,T-\delta], x\in B(0,R)} \int_{T-2^{-n}}^{T} \int_{B(0,R+r)}|\nabla_x q(s-t,x,y)| |h(s, y)| \d y \d s \\
& \ \ \ \ + \sup_{t \in [0,T-\delta], x\in B(0,R)} \int_{T-2^{-n}}^{T} \int_{B(0,R+r)^c}|\nabla_x q(s-t,x,y)| |h(s, y)| \d y \d s \\
&\le c (\delta-2^{-n})^{-\frac{d+1}{2}}    \int_{T-2^{-n}}^{T} \int_{B(0,R+r)}  |h(s, y)| \d y \d s \\
& \ \ \ \ +c \sup_{t \in [0,T-\delta], x\in B(0,R)} \int_{T-2^{-n}}^{T} \int_{B(0,R+r)^c} (s-t)^{-\frac{d+1}{2}} e^{-\frac{|x-y|^2}{4(s-t)}} |h(s, y)| \d y \d s \\
&\le c (\delta-2^{-n})^{-\frac{d+1}{2}}   \int_{T-2^{-n}}^{T} \int_{B(0,R+r)}  |h(s, y)| \d y \d s \\
& \ \ \ \ +c e^{-\frac{r^2}{8T}}\sup_{t \in [0,T-\delta], x\in B(0,R)} \int_{T-2^{-n}}^{T} \int_{B(0,R+r)^c} (s-t)^{-\frac{d+1}{2}} e^{-\frac{|x-y|^2}{8(s-t)}} |h(s, y)| \d y \d s \\
&\le c (\delta-2^{-n})^{-\frac{d+1}{2}}   \int_{T-2^{-n}}^{T} \int_{B(0,R+r)}  |h(s, y)| \d y \d s \\
& \ \ \ \ +c e^{-\frac{r^2}{8T}}\sup_{t \in [0,T-\delta], x\in B(0,R)} \int_{t}^{T} \int_{\Rd} (s-t)^{-\frac{d+1}{2}} e^{-\frac{|x-y|^2}{8(s-t)}} |h(s, y)| \d y \d s.
\end{split}
\end{eqnarray}
 From the definition of the Kato class and Lemma 2.1, we see that $  |h(s, y)| $ is integrable on $(0, T) \times B(0,R+r)$ and
 $$\sup_{t \in [0,T-\delta], x\in B(0,R)} \int_{t}^{T} \int_{\Rd} (s-t)^{-\frac{d+1}{2}} e^{-\frac{|x-y|^2}{8(s-t)}} |h(s, y)| \d y \d s<\infty.$$
 Now first letting $n \to \infty$ and then letting $r \to \infty$ in (\ref{3.7-1}), we obtain (\ref{3.6-1}).

\vspace{3mm}

Now we show (\ref{2.20}).
Note that for $\alpha \in (0,1)$, there exists constants $C_1>0$ and $0<C_2\le 1$ such that for $0<t_1<t_2 \le T$,
\begin{eqnarray}   \label{3.8-1}
\begin{split}
 &\ \ \ \    |\nabla_x q(t_1,x_1,y)-\nabla_x q(t_2,x_2,y)|\\
&\le C_1 |x_1-x_2|^{\alpha}t_1^{-\frac{d+1+\alpha}{2}}(e^{-\frac{C_2|x_1-y|^2}{2t_1}}+e^{-\frac{C_2|x_2-y|^2}{2t_1}})+C_1|t_1-t_2|^{\alpha}t_1^{-\frac{d+1+2\alpha}{2}} e^{-\frac{C_2 |x_2-y|^2}{2t_2}}.
\end{split}
\end{eqnarray}
Take $\delta \in (0,T )$ and a large positive integer $N$ so that $2^{-N} <\delta$. For $t \in [0,T-\delta]$, since $\text{supp} \ \varphi_n \subset (0,2^{-n}) \times B(0,2^{-n})$,
we have for $n \ge N$,
\begin{eqnarray}  \label{2.21}
\begin{split}
&\ \ \ \  | \int_t^{T}\int_{\Rd}  \nabla_x  q(s-t,x,y) f_n(s,y)( h_n (s,y)-h (s,y) ) \d y \d s| \\
&\le \|f_n\|_{ L^{\infty} }  \int_t^{t+\delta}\int_{\Rd}  |\nabla_x  q(s-t,x,y)| ( |h_n (s,y)|+|h (s,y)| ) \d y \d s\\
&\ \ \ \  +\|f_n\|_{ L^{\infty} }  \!  \! \int_{\R^{d+1}}  \!  \varphi_n( \tau,z) \d z \d \tau \int_{t+\delta-\tau}^{T-\tau}  \! \int_{\Rd}  \!  |\nabla_xq(s-t+\tau,x-z,y) \! - \! \nabla_x q(s-t,x,y)| |h(s, y)| \d y \d s\\
&\ \ \ \  + \int_{\R^{d+1}} \varphi_n( \tau,z) \d z \d \tau \int_{t+\delta-\tau}^{T-\tau} \int_{\Rd} |\nabla_xq(s-t,x,y)| |f_n(s+\tau, y+z)-f_n(s,y) | |h(s, y)| \d y \d s\\
&\ \ \ \  + \|f_n\|_{ L^{\infty} }  \sup_{\tau \in (0,2^{-n})}\int_{t+\delta-\tau }^{t+\delta} \int_{\Rd} |\nabla_x q(s-t,x,y)| |h(s, y)| \d y \d s\\
&\ \ \ \  + \|f_n\|_{ L^{\infty} }  \sup_{\tau \in (0,2^{-n})} \int_{T-\tau}^{T} \int_{\Rd}|\nabla_x q(s-t,x,y)| |h(s, y)| \d y \d s\\
&\le c \mathcal{N}^{1}_h(\delta)+c\int_{\R^{d+1}} \varphi_n( \tau,z) \d z \d \tau \int_{t+\delta-\tau}^{T-\tau} \int_{\Rd} [|z|^{\alpha} (s-t)^{-\frac{d+1+\alpha}{2}}(e^{-\frac{C_2|x-z-y|^2}{2(s-t)}}+e^{-\frac{C_2|x-y|^2}{2(s-t)}})\\
&\ \ \ \  + \tau^{\alpha}(s-t)^{-\frac{d+1+2\alpha}{2}} e^{-\frac{C_2|x-z-y|^2}{2(s-t+\tau)}} ]|h(s, y)| \d y \d s \\
&\ \ \ \  + \int_{\R^{d+1}} \varphi_n( \tau,z) \d z \d \tau \int_{t+\delta-\tau}^{T-\tau} \int_{\Rd} |\nabla_xq(s-t,x,y)| |f_n(s+\tau, y+z)-f_n(s,y) | |h(s, y)| \d y \d s \\
&\ \ \ \  +  c \int_{t}^{t+\delta} \int_{\Rd} |\nabla_x q(s-t,x,y)| |h(s, y)| \d y \d s \\
&\ \ \ \  +  c \int_{T-2^{-n}}^{T} \int_{\Rd}|\nabla_x q(s-t,x,y)| |h(s, y)| \d y \d s \\
&\le c \mathcal{N}^{1}_h(\delta)+c2^{-n \a} [(\delta-2^{-n})^{-\frac{d+1+\alpha}{2}} +(\delta-2^{-n})^{-\frac{d+1+2\alpha}{2}})] \sup_{x \in \Rd} \int_{\delta -2^{-N}}^{T} \int_{\Rd}  e^{-\frac{C_2|x-y|^2}{2T}} |h(s, y)| \d y \d s \\
&\ \ \ \  +  c  \sup_{ \tau , |z| \in [0,2^{-n}) } \sup_{(s,y) \in [0, T-\tau] \times \Rd } |f_n(s+\tau, y+z)-f_n(s,y) | \\
&\ \ \ \   \times\sup_{t \in [0, T-\delta]}\int_{t}^{T} \int_{\Rd}  |\nabla_x q(s-t,x,y)| |h(s, y)|  \d y \d s \\
&\ \ \ \  +  \sup_{t \in [0,T-\delta]} \int_{T-2^{-n}}^{T} \int_{\Rd}|\nabla_x q(s-t,x,y)| |h(s, y)| \d y \d s.
\end{split}
\end{eqnarray}
On the other hand, for $t \in [T-\delta,T]$,
\begin{eqnarray}  \label{3.11-1}
\begin{split}
&\ \ \ \  | \int_t^{T}\int_{\Rd}  \nabla_x  q(s-t,x,y) f_n(s,y)( h_n (s,y)-h (s,y) ) \d y \d s| \\
&\le \|f_n\|_{ L^{\infty} }  \int_0^{T-t}\int_{\Rd}  |\nabla_x  q(s,x,y)| ( |h_n (s+t,y)|+|h (s+t,y)| ) \d y \d s \le c \mathcal{N}^{1}_h(\delta).
\end{split}
\end{eqnarray}
(\ref{2.21}) and (\ref{3.11-1}) together  implies that for all $t \in [0,T]$ and $n \ge N$,
\begin{eqnarray}  \label{3.12-1}
\begin{split}
&\ \ \ \  | \int_t^{T}\int_{\Rd}  \nabla_x  q(s-t,x,y) f_n(s,y)( h_n (s,y)-h (s,y) ) \d y \d s | \\
&\le c \mathcal{N}^{1}_h(\delta)+c2^{-n \a} [(\delta \!   - \! 2^{-n})^{-\frac{d \! + \! 1 \! + \! \alpha}{2}} \! \! +    \! (\delta \! - \! 2^{-n})^{-\frac{d \! + \! 1 \! + \! 2\alpha}{2}})] \sup_{x \in \Rd} \int_{\delta -2^{-N}}^{T} \! \! \int_{\Rd}  e^{-\frac{C_2|x-y|^2}{2T}} |h(s, y)| \d y \d s \\
&\ \ \ \  +  c  \sup_{ \tau , |z| \in [0,2^{-n}) } \sup_{(s,y) \in [0, T-\tau] \times \Rd } |f_n(s+\tau, y+z)-f_n(s,y) | \\
&\ \ \ \   \times\sup_{t \in [0, T-\delta]}\int_{t}^{T} \int_{\Rd}  |\nabla_x q(s-t,x,y)| |h(s, y)|  \d y \d s \\
&\ \ \ \  +  \sup_{t \in [0,T-\delta]} \int_{T-2^{-n}}^{T} \int_{\Rd}|\nabla_x q(s-t,x,y)| |h(s, y)| \d y \d s.
\end{split}
\end{eqnarray}
Since $h \in \K_{d,1}$, $\sup_{x \in \Rd} \int_{\delta -2^{-N}}^{T} \int_{B(x,1)}  |h(s, y)| \d y \d s<\infty$. Combining with Lemma \ref{L2.2-1} we see that
\begin{eqnarray}  \label{3.13-1} \nonumber
\begin{split}
 \sup_{x \in \Rd} \int_{\delta -2^{-N}}^{T} \int_{\Rd}  e^{-\frac{C_2|x-y|^2}{2T}} |h(s, y)| \d y \d s < \infty.
\end{split}
\end{eqnarray}
Hence, by (\ref{3.2-1}) and (\ref{3.6-1}), first letting $n \to \infty$ and then letting $\delta \to 0$ in (\ref{3.12-1}), we obtain (\ref{2.20}).
\end{proof}

\vspace{4mm}

 For a matrix $A=(a_{ij})_{1\le i,j \le d}$, define $\| A \|:=\sup_{1\le i,j\le d}|a_{ij}|$. Introduce
\begin{eqnarray}  \label{2.55} \nonumber
\begin{split}
& b_n(t,x):=\int_{\R^{d+1}} \varphi_n(t-s,x-y)b(s, y)\d y \d s, \\
& f_n(t,x):=\int_{\R^{d+1}} \varphi_n(t-s,x-y)f(s, y)\d y \d s .
\end{split}
\end{eqnarray}
For $b, f \in \K_{d,1}$, it is easy to see that $b_n$ and $f_n$ are bounded, smooth functions.
Thus by \cite[Corrollary VI.4.2 and Theorem VI.4.6]{Friedman1}, for any $T>0$, there exists a unique solution $u_n \in C^{1,2}([0,T]\times \Rd)$ to the following parabolic equations:
\begin{align} \label{2.17}
\left\{
\begin{aligned}
&\partial_t u_n (t,x)+\frac{1}{2}\Delta u_n(t,x)+\< b_n(t,x),  \nabla_x u_n(t,x) \>   = f_n(t,x),&&\forall (t,x) \in (0,T)\times \Rd, \\
& u_n(T,x) = 0, && \forall x \in \Rd.
\end{aligned}
\right.
\end{align}
Furthermore, by \cite[Theorem VI.4.5]{Friedman1} and the boundedness of $f_n$, it is easy to see that $u_n \in C^{0,1}_b([0,T]\times \Rd)$.
\vskip 0.4cm

Now we have the following result on the existence and uniqueness of mild solution to the equation (\ref{2.13}):

\begin{prp} \label{L3.3}
Assume that $|b|$ and $|f|$ belong to the Kato class $\K_{d,1}$. Then there exists a constant $T>0$, which depend only on the functions $\mathcal{N}^{1}_b(\cdot)$ and $\mathcal{N}^{1}_f(\cdot)$, such that the following statements hold:

$(\mathrm{i})$ there exists a unique mild solution $u\in C_b^{0,1}([0,T]\times \Rd)$ to the equation (\ref{2.13}).

$(\mathrm{ii})$ for any $R>0$,
\begin{eqnarray}
&&   \sup_{n\ge 1} \|u_n\|_{C_b^{0,1}([0,T]\times \Rd)}\vee\|u\|_{C_b^{0,1}([0,T]\times \Rd)}  \le M_3 \mathcal{N}^{1}_f(T), \label{2.12-2}\\
&&  \lim_{n \to \infty}     \|u_n -u \|_{C_b^{0,1}( [0,T] \times B(0,R))}=0, \label{2.12}
\end{eqnarray}
where $M_3$ is some positve constant depending only on $d$ and $\mathcal{N}^{1}_b(T)$.

$(\mathrm{iii})$
\begin{eqnarray}  \label{2.22} \nonumber
\begin{split}
 |u(t,y)-u(t,x)|\le \frac{1}{2} |x-y|, \ \forall (t,x),(t,y) \in [0,T]\times \Rd.
\end{split}
\end{eqnarray}
\end{prp}
\begin{proof}
First we show (\ref{2.12-2}). Assume $u$ is a mild solution to equation (\ref{2.13}). Then
\begin{eqnarray}  \label{2.15} \nonumber
\begin{split}
u(t,x)=\int_t^T\int_{\Rd} q(s-t,x,y) [\< b(s,y),  \nabla_x u(s,y) \>-f(s,y)] \d y \d s.
\end{split}
\end{eqnarray}
Hence
\begin{eqnarray}  \label{2.14}
\begin{split}
|\nabla_x u(t,x)|&\le \int_t^T\int_{\Rd} |\nabla_x  q(s-t,x,y)|  | \< b(s,y),  \nabla_x u(s,y) \>-f(s,y) | \d y \d s \\
&\le  \int_t^T\int_{\Rd} |\nabla_x  q(s-t,x,y)|  |  b(s,y)  | \d y \d s  \sup_{(s,y) \in [0,T] \times \Rd}\|\nabla_x u(s,y) \|\\
&\ \ \ \ + \int_t^T\int_{\Rd} |\nabla_x  q(s-t,x,y)|  |  f(s,y) | \d y \d s\\
&\le  C_0 \mathcal{N}^{1}_b(T)   \sup_{(s,y) \in [0,T] \times \Rd}\|\nabla_x u(s,y) \|+C_0 \mathcal{N}^{1}_f(T),
\end{split}
\end{eqnarray}
for some positive constant $C_0$ depending only on $d$. Since $|b|, |f| \in \K_{d,1}$, take $T>0$ sufficiently small so that
\begin{equation}\label{2.16}
C_0 \mathcal{N}^{1}_b(T)  \le \frac{1}{2}, \  C_0 \mathcal{N}^{1}_f(T)  \le \frac{1}{2d}.
\end{equation}
Combining (\ref{2.14}) with (\ref{2.16}) yields that
\begin{eqnarray}  \label{2.18}
\begin{split}
 \sup_{(t,x) \in [0,T] \times \Rd}|\nabla_x u(t,x)|\le 2 C_0 \mathcal{N}^{1}_f(T).
\end{split}
\end{eqnarray}
By (\ref{2.15}) and (\ref{2.18}),
\begin{eqnarray}   \label{2.56}
\begin{split}
 \sup_{(t,x) \in [0,T] \times \Rd} |u(t,x)| &=\int_t^T\int_{\Rd} q(s-t,x,y) [\< b(s,y),  \nabla_x u(s,y) \>-f(s,y)] \d y \d s\\
 &\le c  \mathcal{N}^{1}_b(T) \mathcal{N}^{1}_f(T)+c \mathcal{N}^{1}_f(T) \le C_1 \mathcal{N}^{1}_f(T),
\end{split}
\end{eqnarray}
$C_1$ is some positve constant dependent only on $d$ and $\mathcal{N}^{1}_b(T)$.

By Lemma \ref{L3.2} and (\ref{2.16}), we have
\begin{eqnarray}  \label{3.23-1}
\begin{split}
\sup_{n \ge 1} C_0 \mathcal{N}^{1}_{b_n}(T)\le C_0 \mathcal{N}^{1}_b(T)\le \frac{1}{2}.
\end{split}
\end{eqnarray}
Following the proofs for  (\ref{2.18}) and (\ref{2.56}), we have for $n \ge 1$,
\begin{eqnarray} \nonumber
\begin{split}
 \sup_{(t,x) \in [0,T] \times \Rd}|\nabla_x u_n (t,x)| \vee \sup_{(t,x) \in [0,T] \times \Rd} |u_n(t,x)| &\le (2C_0\vee C_1) \mathcal{N}^{1}_f(T).
\end{split}
\end{eqnarray}
Hence (\ref{2.12-2}) holds.

\vskip 0.3cm

It is easy to see that the uniqueness of the mild solution to the equation (\ref{2.13}) follows from (\ref{2.12-2}). On the other  hand, since $u_n$ is the solution to the equation (\ref{2.17}), (\ref{2.12}) would imply  that $u$ is the unique mild solution to the equation (\ref{2.13}). Therefore to prove $(\mathrm{i})$ we only need to show that $\{u_n\}_{ n\geq 1 }$  is a Cauchy sequence in $C_b^{0,1}( [0,T] \times B(0,R))$ for any $R>0$.

\vskip 0.3cm

Recall that $u_n$ is the mild solution to the equation (\ref{2.17}) satisfying
\begin{eqnarray}  \nonumber
\begin{split}
u_n(t,x)=\int_t^T\int_{\Rd} q(s-t,x,y) [\< b_n (s,y),  \nabla_x u_n (s,y) \>-f_n(s,y)] \d y \d s.
\end{split}
\end{eqnarray}
For $m,n>0$, by Lemma \ref{L3.2} and (\ref{3.23-1}) we see that
\begin{eqnarray}  \nonumber
\begin{split}
 &\ \ \ \  |\nabla_x u_m(t,x)- \nabla_x u_n(t,x)|\\
&\le \int_t^T\int_{\Rd}| \nabla_x  q(s-t,x,y)| | b_n (s,y)|  \d y \d s \sup_{ (s,y) \in [0,T] \times \Rd} |\nabla_x u_n (s,y)- \nabla_x u_m(s,y)| \\
&\ \ \ \ + | \int_t^T\int_{\Rd}  \nabla_x  q(s-t,x,y)  \< \nabla_x u_m(s,y),  b_n (s,y)- b_m (s,y)  \> \d y \d s | \\
&\ \ \ \ + | \int_t^T\int_{\Rd}  \nabla_x  q(s-t,x,y)   ( f_n(s,y)-  f_m(s,y)  )   \>\d y \d s |\\
&\le \frac{1}{2} \sup_{ (s,y) \in [0,T] \times \Rd} |\nabla_x u_n (s,y)- \nabla_x u_m(s,y)| \\
&\ \ \ \ + | \int_t^T\int_{\Rd}  \nabla_x  q(s-t,x,y)  \< \nabla_x u_m(s,y),  b_n (s,y)- b_m (s,y)  \> \d y \d s | \\
&\ \ \ \ + | \int_t^T\int_{\Rd}  \nabla_x  q(s-t,x,y)   ( f_n(s,y)-  f_m(s,y)  )  \d y \d s |.
\end{split}
\end{eqnarray}
Thus
\begin{eqnarray}  \label{3.19-1}
\begin{split}
 &\ \ \ \   |\nabla_x u_m(t,x)- \nabla_x u_n(t,x)|\\
&\le  2 | \int_t^T\int_{\Rd}  \nabla_x  q(s-t,x,y)  \< \nabla_x u_m(s,y),  b_n (s,y)- b_m (s,y)  \> \d y \d s | \\
&\ \ \ \ + 2 | \int_t^T\int_{\Rd}  \nabla_x  q(s-t,x,y)   ( f_n(s,y)-  f_m(s,y)  )   \d y \d s |.
\end{split}
\end{eqnarray}

On the other hand, setting $F_n(s,y):=\< b_n (s,y),  \nabla_x u_n (s,y) \>-f_n(s,y)$, by (\ref{3.8-1}), for $0\le t_1<t_2 \le T$, $x_1,x_2 \in \Rd$, $\delta \in (0,T)$ and $\alpha \in (0,1)$, we have
\begin{eqnarray}  \nonumber
\begin{split}
 &\ \ \ \  |\nabla_x u_n(t_1,x_1)- \nabla_x u_n(t_2,x_2)|\\
&\le \int_{t_1}^{t_2}\int_{\Rd}| \nabla_x  q(s-t_1,x_1,y)| | F_n (s,y)|  \d y \d s \\
&\ \ \ \ +\int_{t_2}^{T}\int_{\Rd}| \nabla_x  q(s-t_1,x_1,y)-\nabla_x  q(s-t_2,x_2,y)| | F_n (s,y)|  \d y \d s  \\
&\le c \mathcal{N}^{1}_{F_n}(t_2-t_1) + \int_{t_2}^{t_2+\delta}\int_{\Rd}| \nabla_x  q(s-t_1,x_1,y)| | F_n (s,y)|  \d y \d s  \\
&\ \ \ \ +\int_{t_2}^{t_2+\delta}\int_{\Rd}| \nabla_x  q(s-t_2,x_2,y)| | F_n (s,y)|  \d y \d s  \\
&\ \ \ \ +c \int_{t_2+\delta}^{T}\int_{\Rd} |x_1-x_2|^{\alpha} (s-t_2)^{-\frac{d+1+\alpha}{2}}(e^{-\frac{C_2|x_1-y|^2}{2(s-t_2)}}+e^{-\frac{C_2|x_2-y|^2}{2(s-t_2)}} )| F_n (s,y)|  \d y \d s \\
&\ \ \ \ +c \int_{t_2+\delta}^{T}\int_{\Rd}|t_1-t_2|^{\alpha}(s-t_2)^{-\frac{d+1+2\alpha}{2}}  e^{-\frac{C_2|x_2-y|^2}{2(s-t_1)}} | F_n (s,y)|  \d y \d s\\
&\le  c \mathcal{N}^{1}_{F_n}(t_2-t_1)+ \int_{t_1}^{t_2+\delta}\int_{\Rd}| \nabla_x  q(s-t_1,x_1,y)| | F_n (s,y)|  \d y \d s  + c \mathcal{N}^{1}_{F_n}(\delta)  \\
&\ \ \ \ +c( |x_1-x_2|^{\alpha} \delta^{-\frac{d+1+\alpha}{2}} +|t_1-t_2|^{\alpha} \delta^{-\frac{d+1+2\alpha}{2}} )\sup_{x \in \Rd} \int_{\delta}^{T}\int_{\Rd} e^{-\frac{C_2|x-y|^2}{2T}}  | F_n (s+t_1,y)|  \d y \d s\\
&\le  c \mathcal{N}^{1}_{F_n}(t_2-t_1+\delta)+ c( |x_1-x_2|^{\alpha} \delta^{-\frac{d+1+\alpha}{2}} + |t_1-t_2|^{\alpha} \delta^{-\frac{d+1+2\alpha}{2}} ),
\end{split}
\end{eqnarray}
where the last inequality follows from Lemma \ref{L2.2-1} and the fact that $$\sup_{t \ge 0,x\in \Rd,n \ge 1} \int_{\delta}^T\int_{B(x,1)} |F_n(s+t,y)| \d y \d s<\infty.$$
Hence (\ref{3.2-1}) is fulfilled by $\nabla_x u_n$.
Now using Lemma \ref{L3.2} and (\ref{3.19-1}), we see that $\{u_n\}_{ n\geq 1 }$ is a Cauchy sequence in $C_b^{0,1}( [0,T] \times B(0,R))$.

\vskip 0.3cm

$(\mathrm{iii})$ follows from (\ref{2.16}) and (\ref{2.18}).
\end{proof}

\vskip 0.4cm
Next result gives the H\"o{}lder continuity of $\nabla u$ with respect to the time variable.

\begin{prp} \label{L3.3-1}
Assume
\begin{eqnarray} \label{3.15-1}
\sup_{t \ge 0, x\in \Rd}\int_{0}^{r^2} \int_{B(x,r)}(|b(t+s,y)|+|f(t+s,y)|) \d y \d s \le C r^p,  \ \forall 0<r <1.
\end{eqnarray}
for some constants $p>d+1$ and $C>0$. Then for any $\a \in (0, \frac{p-d-1}{2} \wedge \frac{p-d}{d+3} \wedge 1)$, there exists a $M_4>0$ depending on $d,T,\a,p$ and $C$, such that
\begin{eqnarray}  \label{4.23-1}
\begin{split}
  |\nabla_x u(t_1,x)-\nabla_x u(t_2,x)| \le M_4 |t_1-t_2|^{ \a }, \ \forall (t_1,x),(t_2,x) \in [0,T] \times \Rd.
\end{split}
\end{eqnarray}
\end{prp}
\begin{proof}
 Assume (\ref{3.15-1}) holds. Then by Proposition \ref{P2.3}, $|b|$ and $|f|$ belong to the Kato class $\K_{d,1}$. Take a  constant $\a \in (0, \frac{p-d-1}{2} \wedge \frac{p-d}{d+3} \wedge 1)$.
Set $F(s,y):=\< b  (s,y),  \nabla_x u  (s,y) \>-f (s,y)$. First we prove
\begin{eqnarray}  \label{2.31}
\begin{split}
&\ \ \ \ \sup_{0\le t_1<t_2 \le T, x \in \Rd \atop  t_2-t_1<1}\int_{t_2}^T\int_{\Rd}  (s-t_2)^{-\frac{\a(d+3)}{2}} (s-t_1)^{-\frac{(1-\a)(d+1)}{2}}e^{-\frac{(1-\a)|x-y|^2}{4 (s-t_1)}}   |  F(s,y)  | \d y \d s   <\infty .
\end{split}
\end{eqnarray}

Set $\delta:=t_2-t_1<1$. Since $s-t_1 \in (\delta, 2\delta)$ for $s \in (t_2, t_2+\delta)$ and $s-t_2<s-t_1<2(s-t_2)$ for $s \in (t_2+\delta,T)$, we have
\begin{eqnarray}  \label{2.29}
\begin{split}
&\ \ \ \ \int_{t_2}^T\int_{\Rd}  (s-t_2)^{-\frac{\a(d+3)}{2}} (s-t_1)^{-\frac{(1-\a)(d+1)}{2}}e^{-\frac{(1-\a)|x-y|^2}{4 (s-t_1)}}   |  F(s,y)  | \d y \d s   \\
&\le   \int_{t_2}^{t_2+\delta}\int_{\Rd}  (s-t_2)^{-\frac{\a(d+3)}{2}}  \delta^{-\frac{(1-\a)(d+1)}{2}}e^{-\frac{(1-\a)|x-y|^2}{8 \delta}}   |  F(s,y)  | \d y \d s   \\
&\ \ \ \  + c \int_{(t_2+\delta) \wedge T}^{T }\int_{\Rd}  (s-t_2)^{-\frac{\a(d+3)}{2}} (s-t_2)^{-\frac{(1-\a)(d+1)}{2}}e^{-\frac{(1-\a)|x-y|^2}{4 (s-t_2)}}   |  F(s,y)  | \d y \d s   \\
&\le   
 \delta^{-\frac{(1-\a)(d+1)}{2}}  \int_{0}^{\delta}\int_{\Rd}  s^{-\frac{\a(d+3)}{2}} e^{-\frac{(1-\a) |x-y|^2}{8 \delta}}   |  F(t_2+s,y)  | \d y \d s   \\
&\ \ \ \  + c \int_{t_2}^{T  }\int_{\Rd}  (s-t_2)^{-\frac{d+1+2\a}{2}} e^{-\frac{(1-\a)|x-y|^2}{4(s-t_2)}}   |  F(s,y)  | \d y \d s .
\end{split}
\end{eqnarray}
Take $0<\lambda<1$ with $p-(2- \lambda)d-(2\a +1) >0$. Then by Lemma \ref{L2.2-1} and Lemma \ref{L2.3} we have
\begin{eqnarray}\label{2.29-1}
\begin{split}
&\ \ \ \   \int_{0}^{\delta}\int_{\Rd}  s^{-\frac{\a(d+3)}{2}} e^{-\frac{(1-\a) |x-y|^2}{8 \delta}}   |  F(t_2+s ,y)  | \d y \d s   \\
&=   \sum_{n \ge 0}\int_{2^{-n-1}\delta}^{2^{-n}\delta}\int_{\Rd} s^{-\frac{\a(d+3)}{2}} e^{-\frac{(1-\a) |x-y|^2}{8 \delta}}   |   F(t_2+s,y)  | \d y \d s   \\
&\le   \sum_{n \ge 0} 2^{\frac{\a(n+1)(d+3)}{2} }\delta^{-\frac{\a(d+3)}{2}} (\int_{2^{-n-1}\delta}^{2^{-n}\delta}\int_{|x-y| < \delta^{\frac{\lambda}{2}}}  e^{-\frac{(1-\a) |x-y|^2}{8 \delta}}   |   F(t_2+s,y)  | \d y \d s\\
&\ \ \ \ +\int_{2^{-n-1}\delta}^{2^{-n}\delta}\int_{|x-y| \ge \delta^{\frac{\lambda}{2}}}  e^{-\frac{(1-\a) |x-y|^2}{8 \delta }}   |   F(t_2+s,y)  | \d y \d s )   \\
&\le \sum_{n \ge 0} 2^{\frac{\a(n+1)(d+3)}{2} }\delta^{-\frac{\a(d+3)}{2}} (\int_{2^{-n-1}\delta}^{2^{-n}\delta}\int_{|x-y| < \delta^{\frac{\lambda}{2}}}     |   F(t_2+s,y)  | \d y \d s\\
&\ \ \ \ +e^{-\frac{(1-\a)\delta^{\lambda-1}}{16 }} \int_{2^{-n-1}\delta}^{2^{-n}\delta}\int_{|x-y| \ge \delta^{\frac{\lambda}{2}}}  e^{-\frac{(1-\a) |x-y|^2}{16 T}}   |   F(t_2+s,y)  | \d y \d s )   \\
&\le c  \sum_{n \ge 0} 2^{\frac{\a(n+1)(d+3)}{2} }\delta^{-\frac{\a(d+3)}{2}} ( 2^{-\frac{(n+1)(p-d)}{2} }  \delta^{ \frac{ p-d }{2} }  \delta^{\frac{\lambda d}{2}}  \\
&\ \ \ \ + e^{-\frac{(1-\a)\delta^{\lambda-1}}{16}} \sup_{x \in \Rd} \int_{2^{-n-1}\delta}^{2^{-n}\delta}\int_{ B(x,1)}     |   F(t_2+s,y)  | \d y \d s )   \\
&\le c  \sum_{n \ge 0} 2^{\frac{\a(n+1)(d+3)}{2} }\delta^{-\frac{\a(d+3)}{2}} ( 2^{-\frac{(n+1)(p-d)}{2} }  \delta^{ \frac{ p-d }{2} } \delta^{\frac{\lambda d}{2}} + e^{-\frac{(1-\a)\delta^{\lambda-1}}{16 }}  2^{-\frac{(n+1)(p-d)}{2} }  \delta^{ \frac{ p-d }{2} }  )   \\
&= c  \sum_{n \ge 0} 2^{\frac{(n+1)(d \a+ 3\a -p +d )}{2} }  (     \delta^{\frac{ p-d -d \a-3 \a + \lambda d}{2}} + \delta^{ \frac{ p-d -d \a-3 \a}{2}}e^{-\frac{(1-\a)\delta^{\lambda-1}}{16}}   ).
\end{split}
\end{eqnarray}
Note that $d \a+ 3\a -p +d<0 $. (\ref{2.29-1}) yields  that
\begin{eqnarray}  \label{3.26-1}
\begin{split}
&\ \ \ \    \delta^{-\frac{(1-\a)(d+1)}{2}}  \int_{0}^{\delta}\int_{\Rd}  s^{-\frac{\a(d+3)}{2}} e^{-\frac{(1-\a) |x-y|^2}{8 \delta}}   |  F(t_2+s,y)  | \d y \d s  \\
&\le c  \sum_{n \ge 0} 2^{\frac{(n+1)(d \a+ 3\a -p +d )}{2} } ( \delta^{ \frac{ p-(2- \lambda)d-(2 \a +1) }{2} }+ \delta^{ \frac{ p-d -d \a-3 \a-(1-\a)(d+1)}{2}}e^{-\frac{(1-\a)\delta^{\lambda-1}}{16 }} ) \le c,
\end{split}
\end{eqnarray}
where we used the fact that $\l-1<0$ in the last inequality. By Proposition \ref{P2.3}, (\ref{3.15-1}) and the fact that  $2\a <p-d-1$, we see that $|F| \in \K_{d,1-2\a}$. Therefore (\ref{2.29}) and (\ref{3.26-1}) imply
\begin{eqnarray}  \nonumber
\begin{split}
&\ \ \ \ \int_{t_2}^T\int_{\Rd}  (s-t_2)^{-\frac{\a(d+3)}{2}} (s-t_1)^{-\frac{(1-\a)(d+1)}{2}}e^{-\frac{(1-\a)|x-y|^2}{4(s-t_1)}}   |  F(s,y)  | \d y \d s   \\
&\le  c +c \int_{t_2 }^{T}\int_{\Rd}  (s-t_2)^{-\frac{d+1+2\a}{2}} e^{-\frac{(1-\a)|x-y|^2}{4(s-t_2)}}   |  F(s,y)  | \d y \d s \le c.
\end{split}
\end{eqnarray}
(\ref{2.31}) is proved.
\vskip 0.3cm

Now we show (\ref{4.23-1}). Noting that  for $0<t_1<t_2$,
\begin{eqnarray}  \nonumber
\begin{split}
 &\ \ \ \  |\nabla_x q(t_1,x,y)- \nabla_x q(t_2,x,y)|\\
&\le  |t_1-t_2| \int_0^1 | \partial_t \nabla_x q(t_1+\lambda(t_2-t_1),x, y)|       \d \lambda \le c |t_1-t_2| t_1^{-\frac{d+3}{2}},
\end{split}
\end{eqnarray}
we have
\begin{eqnarray}  \label{2.54}
\begin{split}
 &\ \ \ \  |\nabla_x q(t_1,x,y)- \nabla_x q(t_2,x,y)|\\
&\le  |\nabla_x q(t_1,x,y)- \nabla_x q(t_2,x,y)|^{ \a } |\nabla_x q(t_1,x,y)+\nabla_x q(t_2,x,y)|^{1- \a } \\
&\le c |t_1-t_2|^{ \a } t_1^{-\frac{\a(d+3)}{2}} (t_1^{-\frac{(1-\a)(d+1)}{2}}e^{-\frac{(1-\a)|x-y|^2}{4 t_1}}+t_2^{-\frac{(1-\a)(d+1)}{2}}e^{-\frac{(1-\a)|x-y|^2}{4 t_2}})\\
&\le c |t_1-t_2|^{ \a } (t_1^{-\frac{d+1+2\a}{2}} e^{-\frac{(1-\a)|x-y|^2}{4 t_1}}+t_1^{-\frac{\a(d+3)}{2}}t_2^{-\frac{(1-\a)(d+1)}{2}}e^{-\frac{(1-\a)|x-y|^2}{4 t_2}}).
\end{split}
\end{eqnarray}
For $0\le t_1<t_2 \le T$ with $t_2-t_1<1$, using (\ref{2.31}) and (\ref{2.54}) we get
\begin{eqnarray}  \label{2.30} \nonumber
\begin{split}
&\ \ \ \ |\nabla_x u(t_1,x)-\nabla_x u(t_2,x)|\\
&\le \int_{t_1}^{t_2}\int_{\Rd} |\nabla_x q(s-t_1,x, y)|  | F(s,y) | \d y \d s \\
&\ \ \ \ + \int_{t_2}^T\int_{\Rd} |\nabla_x q(s-t_1,x, y)-\nabla_x q(s-t_2,x, y) |  |  F(s,y)  | \d y \d s  \\
&\le c |t_1-t_2|^{\a} \int_{t_1}^{t_2}\int_{\Rd}  (s-t_1)^{-\frac{d+1+2\a}{2}} e^{-\frac{|x-y|^2}{4(s-t_1)}}   | F(s,y) | \d y \d s \\
&\ \ \ \ +  c |t_1-t_2|^{ \a } \int_{t_2}^T\int_{\Rd}  (s-t_2)^{-\frac{d+1+2\a}{2}} e^{-\frac{(1-\a)|x-y|^2}{4 (s-t_2)}}   |  F(s,y)  | \d y \d s  \\
&\ \ \ \ +  c |t_1-t_2|^{ \a } \int_{t_2}^T\int_{\Rd}  (s-t_2)^{-\frac{\a(d+3)}{2}} (s-t_1)^{-\frac{(1-\a)(d+1)}{2}}e^{-\frac{(1-\a)|x-y|^2}{4 (s-t_1)}}   |  F(s,y)  | \d y \d s \\
&\le c |t_1-t_2|^{ \a }.
\end{split}
\end{eqnarray}
\end{proof}

\begin{remark}
Using H\"{o}lder's inequality, it is easy to see that  (\ref{3.15-1}) holds if $|b|^2$ and $|f|^2$ satisfy the Hypothesis  $\mathrm{H_p}$.
\end{remark}

\vskip 0.3cm
Next, we will establish some regularities of the mild solution of equation (\ref{2.13}), which plays a curial role later. We start with the  following  $L^2$-estimate.

\begin{lem}  \label{L3.4}
 Suppose a function $ g:[0,T] \times \Rd \to \R$ vanishes outside  $[0,T] \times B(x_0,r)$ for some $r \in (0,2)$, and moreover $|g|^2$ satisfies the Hypothesis  $\mathrm{H_p}$. Let $u$ be the mild solution of the following equation:
\begin{align} \nonumber
\left\{
\begin{aligned}
&\partial_t u (t,x)+\frac{1}{2}\Delta u  (t,x)  = g(t,x),&&\forall (t,x) \in (0,T)\times \Rd, \\
& u  (T,x) = 0 , && \forall x \in \Rd.
\end{aligned}
\right.
\end{align}
Then $u\in W^{1,2}_2((0,T)\times \Rd)$. Furthermore, there exist constants  $\a,\b \in (0,p-d)$ and $M_5>0$, which are independent of $x_0$ and $r$, such that
for any $0\le a<b  \le T$,
\begin{eqnarray}  \label{2.52}
\begin{split}
\int_{a}^{b} \int_{\Rd} \| D^2_xu  \|^2 \d x \d t  \le M_5  (\int_{a}^{b} \int_{B(x_0,r)}|   g  (t,x) |^2 \d x\d t +r^{d}(b-a)^{\a}  + r^{p-\b}).
\end{split}
\end{eqnarray}
\end{lem}
\begin{proof}
Let $g_n(t,x):=\int_{\R^{d+1}} \varphi_n(t-s,x-y) g(s, y)\d y \d s$ and $u_n \in C_b^{1,2}([0,T]\times \Rd)$ be the mild solution of the following equation:
\begin{align} \label{2.32} \nonumber
\left\{
\begin{aligned}
&\partial_t u_n (t,x)+\frac{1}{2}\Delta u_n  (t,x)  = g_n(t,x),&&\forall (t,x) \in (0,T)\times \Rd, \\
& u_n (T,x) = 0 , && \forall x \in \Rd.
\end{aligned}
\right.
\end{align}
Using the representation
\begin{eqnarray}  \nonumber
\begin{split}
u_n(t,x)=\int_t^T \int_{B(x_0,r+1)} q(s-t,x,y)  g_n(s,y) \d y \d s,
\end{split}
\end{eqnarray}
we can see easily that for any $R \ge 1$ and $|x-x_0|\ge R+r+1$,
\begin{eqnarray}  \label{2.33}
\begin{split}
\sup_{n \ge 1, t\ge 0} \{\|D^2_x u_n(t,x)\|+|\nabla_xu_n(t,x)|+|\partial_t u_n(t,x)|\} \le c e^{-\frac{R^2}{4T}}.
\end{split}
\end{eqnarray}
Using the integration by parts  we have
\begin{eqnarray}  \label{3.35-1}
\begin{split}
\int_{\Rd} \|  g_n(t,x) \|^2 \d x  =  \int_{\Rd} (|\partial_t u_n(t,x)|^2+\frac{1}{4}|\Delta u_n(t,x)|^2) \d x  -\int_{\Rd}  \partial_t |\nabla_xu_n(t,x)|^2 \d x.
\end{split}
\end{eqnarray}
and
\begin{eqnarray}  \label{2.35}
\begin{split}
& \ \ \ \  \int_{\Rd} \|  g_n(t,x)- g_m(t,x) \|^2 \d x   \\
&=  \int_{\Rd} (|\partial_t u_n(t,x)-\partial_t u_m(t,x)|^2+\frac{1}{4}|\Delta u_n(t,x)-\Delta u_m(t,x)|^2) \d x\\
& \ \ \ \  -\int_{\Rd}  \partial_t |\nabla_xu_n(t,x)-\nabla_xu_m(t,x)|^2 \d x. \\
\end{split}
\end{eqnarray}
On the other hand, (\ref{2.33}) and the Green's formula give that
\begin{eqnarray}  \label{2.34}
\begin{split}
\sum_{i,j=1}^{d}  \int_{\Rd} | \partial_{ij}^2 u_n(t,x)-\partial_{ij}^2 u_m(t,x)  |^2 \d x =   \int_{\Rd}  | \Delta u_n(t,x)-\Delta u_m(t,x)  |^2 \d x .
\end{split}
\end{eqnarray}
Hence by (\ref{3.35-1})-(\ref{2.34}), we have
\begin{eqnarray}  \label{2.49}
\begin{split}
& \ \ \ \ \sum_{i,j=1}^{d}\int_{0}^{T} \int_{\Rd}  | \partial_{ij}^2 u_n(t,x)-\partial_{ij}^2 u_m(t,x) |^2 \d x \d t +\int_{0}^{T} \int_{\Rd} |\partial_t u_n(t,x)-\partial_t u_m(t,x)|^2 \d x \d t \\
& \le 4 \int_{0}^{T} \int_{\Rd}|   g_n (t,x)-   g_m (t,x)|^2 \d x\d t ,
\end{split}
\end{eqnarray}
and
\begin{eqnarray} \label{2.36}
\begin{split}
& \ \ \ \ \sum_{i,j=1}^{d}\int_{a}^{b} \int_{\Rd}  | \partial_{ij}^2 u_n(t,x)|^2 \d x \d t +\int_{a}^{b} \int_{\Rd}  |\partial_t u_n(t,x) |^2 \d x \d t \\
& \le 4 \int_{a}^{b} \int_{\Rd}|   g_n (t,x)|^2 \d x\d t +4\int_{\Rd} (|\nabla_{x}  u_n (b,x) |^2-|\nabla_{x}   u_n (a,x) |^2) \d x ,
\end{split}
\end{eqnarray}
for any $0\le a < b \le T$.

 Recall that $u_n$ converges to $u$ uniformly on the compact set of $[0,T]\times \Rd$ by Proposition \ref{L3.3}. Hence (\ref{2.49}) implies that $u \in W^{1,2}_2( (0,T)\times \Rd)$ and
\begin{eqnarray}  \label{2.26}
\begin{split}
\lim_{n \to \infty}\int_{a}^{b} \int_{\Rd} \| D^2_xu_n(t,x)-D^2_xu(t,x) \|^2 \d x \d t =0.
\end{split}
\end{eqnarray}
By Proposition \ref{L3.3} and  (\ref{2.33}) we also have
\begin{eqnarray}  \label{3.34-1}
\begin{split}
\lim_{n \to \infty}  \int_{\Rd}   |\nabla_{x}  u_n(t,x)- \nabla_{x} u(t,x) |^2 \d x =0, \ \forall t \in [0,T].
\end{split}
\end{eqnarray}
Thus, by (\ref{2.26}) and (\ref{3.34-1}), letting $n \to \infty$ in (\ref{2.36}) we get
\begin{eqnarray}  \label{2.50}
\begin{split}
& \ \ \ \ \sum_{i,j=1}^{d}\int_{a}^{b} \int_{\Rd} | \partial_{ij}^2 u (t,x) |^2 \d x \d t   \\
& \le 4(\int_{a}^{b} \int_{\Rd}|   g  (t,x) |^2 \d x\d t +\int_{\Rd} (|\nabla_{x}  u  (b,x)|^2 -|\nabla_{x}   u  (a,x) |^2) \d x )\\
& \le  4(\int_{a}^{b}  \int_{B(x_0,r)}    | g  (t,x) |^2 \d x\d t +\int_{|x-x_0|\ge 2r}  |\nabla_{x}  u  (b,x) |^2   \d x \\
&\ \ \ \  +\int_{B(x_0,2r)}   |\nabla_{x}  u  (b,x) -\nabla_{x}   u  (a,x) |(|\nabla_{x}  u  (b,x)|+ |\nabla_{x}   u  (a,x) |)  \d x )\\
& \le c (\int_{a}^{b} \int_{B(x_0,r)} | g  (t,x) |^2 \d x\d t     + \int_{|x-x_0|\ge 2r}  |\nabla_{x}  u  (b,x)|^2 \d x +(b-a)^{\a}r^d),
\end{split}
\end{eqnarray}
where the last inequality follows from Proposition \ref{L3.3} and Proposition \ref{L3.3-1}, and $\a \in (0,p-d)$ is the constant appeared in Proposition \ref{L3.3-1}.

On the other hand, using H\"{o}lder's inequality and the fact that $|g|^2$ satisfies the Hypothesis  $\mathrm{H_p}$, it is easy to see that $g$ satisfies the Hypothesis $(H_{q}')$ with $q:=\frac{p+d+2}{2}>d+1$. Therefore we have for  $t \in [0,T]$,
\begin{eqnarray}  \label{2.51}
\begin{split}
& \ \ \ \ \int_{|x-x_0|\ge 2r}  |\nabla_{x}  u  (t,x)|^2  \d x  \\
&  \le \int_{ |x-x_0|\ge 2r } \int_t^T\int_{|y-x_0|<r} |\nabla_{x}  q(s-t,x,y)| | g(s,y)|  \d y \d s \\
&\ \ \ \ \times \int_t^T\int_{|z-x_0|<r} |\nabla_{x}  q(\tau-t,x,z)| | g(\tau ,z)|  \d z \d \tau \d x  \\
&  \le c  \int_t^T \int_t^T \int_{|y-x_0|<r} \int_{|z-x_0|<r} (s-t)^{-\frac{ 1}{2}} e^{-\frac{r^2}{8(s-t)}} (\tau-t)^{-\frac{ 1}{2}} e^{-\frac{r^2}{8(\tau-t)}} | g(s,y)| | g(\tau ,z)|   \d z \d y  \d \tau \d s \\
& \ \ \ \ \times  \int_{ |x-x_0|\ge 2r } (s-t)^{-\frac{d }{2}}  (\tau -t)^{-\frac{d }{2}} e^{-\frac{|x-y|^2}{8(s-t)}}  e^{-\frac{|x-z|^2}{8(\tau-t)}}  \d x    \\
&  \le c  \int_0^{T-t} \int_0^{T-t} \int_{|y-x_0|<r} \int_{|z-x_0|<r} s^{-\frac{ 1}{2}} e^{-\frac{r^2}{8s}} \tau^{-\frac{ 1}{2}} e^{-\frac{r^2}{8\tau}} (s+\tau)^{-\frac{d }{2}} e^{-\frac{|y-z|^2}{8(s+\tau)}}  \\
&\ \ \ \ \times | g(s+t,y)| | g(\tau +t ,z)|   \d z \d y  \d \tau \d s \\
&  \le c  \int_0^{T-t} \int_{|y-x_0|<r} s^{-\frac{ 1}{2}} e^{-\frac{r^2}{8s}}  | g(s+t,y)| \d y \d s  \int_s^{T-t} \int_{|z-x_0|<r} \tau^{-\frac{ 1}{2}} e^{-\frac{r^2}{8\tau}} (s+\tau)^{-\frac{d }{2}}     | g(\tau +t ,z)|   \d z  \d \tau  \\
&  \le c  \int_0^{r^2} \! \! \int_{|y-x_0|<r} s^{-\frac{ 1}{2}} e^{-\frac{r^2}{8s}}  | g(s+t,y)| \d y \d s  \int_s^{(T-t) \vee r^2} \int_{|z-x_0|<r} \tau^{-\frac{ 1}{2}} e^{-\frac{r^2}{8\tau}} (s+\tau)^{-\frac{d }{2}}     | g(\tau +t ,z)|   \d z  \d \tau  \\
& \ \ \ \ + c  \int_{r^2}^{ (T-t) \vee r^2} \! \! \! \! \int_{|y-x_0|<r} \! \! s^{-\frac{ 1}{2}} e^{-\frac{r^2}{8s}}  | g(s+t,y)| \d y \d s  \int_s^{T-t} \! \! \! \! \int_{|z-x_0|<r} \! \! \tau^{-\frac{ 1}{2}} e^{-\frac{r^2}{8\tau}} (s+\tau)^{-\frac{d }{2}}     | g(\tau +t ,z)|   \d z  \d \tau  \\
&:= \mathrm{I}+\mathrm{II}.
\end{split}
\end{eqnarray}

For  $s \in (0,r^2)$, by (\ref{2.27}) we have
\begin{eqnarray}  \nonumber
\begin{split}
& \ \ \ \ \int_s^{(T-t) \vee r^2} \int_{|z-x_0|<r} \tau^{-\frac{ 1}{2}} e^{-\frac{r^2}{8\tau}} (s+\tau)^{-\frac{d}{2}}     | g(\tau +t ,z)|   \d z  \d \tau  \\
&  \le   \sum_{n \ge 1} \int_{ns}^{(n+1)s} \int_{|z-x_0|<r} \tau^{-\frac{ 1}{2}} (s+\tau)^{-\frac{d }{2}}     | g(\tau +t ,z)|   \d z  \d \tau  \\
&  \le  \sum_{n \ge 1} \int_{ns}^{(n+1)s} \int_{|z-x_0|<r} n^{-\frac{ 1}{2}} (n+1)^{-\frac{d  }{2}}  s^{-\frac{ d+1 }{2}}      | g(\tau +t ,z)|   \d z  \d \tau  \\
&  \le c s^{-\frac{ d+1 }{2}}  s^{\frac{q-d}{2}} r^d \sum_{n \ge 1}   n^{-\frac{ 1}{2}} (n+1)^{-\frac{d  }{2}}  \le c s^{-\frac{ 2d+1 -q}{2}}  r^d.  \\
\end{split}
\end{eqnarray}
Therefore
\begin{eqnarray}  \label{3.40-1}
\begin{split}
\mathrm{I} & \le c r^d \int_0^{r^2}  \int_{|y-x_0|<r} s^{-\frac{ 1}{2}} e^{-\frac{r^2}{8s}} s^{-\frac{ 2d+1 -q}{2}}   | g(s+t,y)| \d y \d s \\
&  \le c r^d r^{-2d-2+q} \int_0^{r^2}  \int_{|y-x_0|<r}  | g(s+t,y)| \d y \d s \\
& \le c r^{-d-2 +2q} = c r^{p } .
\end{split}
\end{eqnarray}
For the term $\mathrm{II}$, we only need to consider the situation: $r^2<T-t$. Take some constant $\b \in (0,p-d)$. For $s \in (r^2, T-t)$, by (\ref{2.28}) we have
\begin{eqnarray}  \nonumber
\begin{split}
& \ \ \ \ \int_s^{T-t} \int_{|z-x_0|<r} \tau^{-\frac{ 1}{2}} e^{-\frac{r^2}{8\tau}} (s+\tau)^{-\frac{d}{2}}     | g(\tau +t ,z)|   \d z  \d \tau  \\
&  \le (2T)^{\frac{\b}{2}} \sum_{n \ge 1} \int_{ns}^{(n+1)s} \int_{|z-x_0|<r} \tau^{-\frac{ 1}{2}} (s+\tau)^{-\frac{d +\b}{2}}     | g(\tau +t ,z)|   \d z  \d \tau  \\
&  \le c \sum_{n \ge 1} n^{-\frac{ 1}{2}} (n+1)^{-\frac{d +\b }{2}}  s^{-\frac{ d+1 + \b}{2}}  \int_{ns}^{(n+1)s} \int_{|z-x_0|<r}      | g(\tau +t ,z)|   \d z  \d \tau  \\
&  \le c s^{-\frac{ d+1+\b }{2}}  s r^{q-2} \sum_{n \ge 1}   n^{-\frac{ 1}{2}} (n+1)^{-\frac{d +\b }{2}}  \le c s^{-\frac{ d-1 +\b}{2}}  r^{q-2}.  \\
\end{split}
\end{eqnarray}
Therefore
\begin{eqnarray}  \label{3.41-1}
\begin{split}
\mathrm{II} & \le c r^{q-2}  \int_{r^2}^{T-t}  \int_{|y-x_0|<r} s^{-\frac{ 1}{2}} e^{-\frac{r^2}{8s}} s^{-\frac{ d-1+ \b}{2}}   | g(s+t,y)| \d y \d s \\
&  \le c r^{q-2-d- \b} \sum_{n \ge 1} \int_{n r^2}^{(n+1) r^2}  n^{-\frac{ d+ \b}{2}} \int_{|y-x_0|<r}  | g(s+t,y)| \d y \d s \\
& \le c r^{2q-2-d- \b} \sum_{n \ge 1}   n^{-\frac{ d+ \b}{2}} = c r^{p - \b} .
\end{split}
\end{eqnarray}
Putting (\ref{2.50})-(\ref{3.41-1}) together, we obtain (\ref{2.52}).
\end{proof}

\vskip 0.3cm

 Now we can state the desired regularity estimate for the solution of the parabolic equation (\ref{2.13}).

\begin{prp} \label{P3.5}
Assume $|b|^2$ and $|f|^2$ satisfy Hypothesis  $\mathrm{H_p}$. Let $u\in C^{0,1}_b([0,T]\times \Rd)$ be the mild solution to the equation (\ref{2.13}). Then
$ \|D^2_xu\|^2 $ satisfies Hypothesis  $\mathrm{H_q}$ for some $q \in (d,p)$.
\end{prp}

\begin{proof}
For any $x_0 \in \Rd$ and $r \in (0,1)$, set $F(t,x):=f(t,x)-\< b(t,x),  \nabla_x u(t,x) \> $ and $F_r(t,x):=F(t,x)   I_{B(x_0,2r)}(x)$. Let $u_r$ be the mild solution to the following equation:
\begin{align} \nonumber
\left\{
\begin{aligned}
&\partial_t u_r (t,x)+\frac{1}{2}\Delta u_r (t,x)  =F_r(t,x),&&\forall (t,x) \in (0,T)\times \Rd, \\
& u_r (T,x) = 0 , && \forall x \in \Rd.
\end{aligned}
\right.
\end{align}
Then by Lemma \ref{L3.4} there exist constants $C>0$ and $q \in (d, p \wedge (d+2))$, which are independent of $x_0,r$, such that any $t \in [0,T]$,
\begin{eqnarray}  \label{2.37}
\begin{split}
\int_{t }^{t +r^2} \int_{B(x_0, r)} \|D^2_x   u_r (s,x)\|^2 \d x \d s & \le C (\int_{t }^{t +r^2} \int_{B(x_0,2r) } |  F_r (s,x)|^2 \d x\d s+r^q) \le C r^q.
\end{split}
\end{eqnarray}

Set $\bar F_r(t,x):=F(t,x)-F_r(t,x)$. Then $\bar u_r:=u-u_r$ is the solution of the following equation
\begin{align}
\left\{
\begin{aligned}
&\partial_t \bar u_r (t,x)+\frac{1}{2}\Delta \bar u_r (t,x)  =\bar F_r(t,x),&&\forall (t,x) \in (0,T)\times \Rd, \\
& \bar u_r (T,x) = 0 , && \forall x \in \Rd.
\end{aligned} \nonumber
\right.
\end{align}
Therefore,
\begin{eqnarray}  \nonumber
\begin{split}
\bar u_r(t,x)=\int_{t }^{T} \int_{\Rd} q(s-t,x,y) \bar F_r(s,y)\d y \d s.
\end{split}
\end{eqnarray}
By Proposition \ref{P2.3} and Lemma \ref{L2.4} it is easy to see that $\bar F_r \in \K_{d,1-\a}$ for $\a  = \frac{q-d}{2}$. Noting
$$\|D_x^2q(t,x,y)\|\le c |x-y|^{ \a-1}t^{-\frac{d+1+\a }{2}}e^{-\frac{|x-y|^2}{4t}},  $$
it follows that $\bar u_r (t,x)$ is twice differentiable with respect to $x$ for $x\in  B(x_0,r)$ and
\begin{eqnarray}  \nonumber
\begin{split}
&\ \ \ \   \|D^2_x  \bar  u_r (t,x)\|   \\
&\le c \int_{t }^{T} \int_{B(x_0,2r)^c} |x-y|^{\a-1 }(s-t)^{-\frac{d+1+\a }{2}}e^{-\frac{|x-y|^2}{4(s-t)}}\bar F_r(s,y)\d y \d s\\
&\le c r^{\a-1 } \int_{t }^{T} \int_{\Rd} (s-t)^{-\frac{d+1+\a }{2}}e^{-\frac{|x-y|^2}{4(s-t)}} \bar F_r(s,y)\d y \d s\le c r^{\a-1 },
\end{split}
\end{eqnarray}
where the last inequality follows from Lemma \ref{L2.2}. Thus
\begin{eqnarray}  \label{2.38}
\begin{split}
\int_{t }^{t +r^2} \int_{B(x_0, r)} \|D^2_x  \bar  u_r (s,x)\|^2 \d x \d s   \le c  r^{d+2\a} =cr^q.
\end{split}
\end{eqnarray}
Hence, by (\ref{2.37}) and (\ref{2.38}), we see that $ \|D^2_xu\|^2 $ satisfies Hypothesis  $\mathrm{H_q}$.
\end{proof}

\section{Existence and uniqueness of  strong solutions}

After the preparations, in this section, we are ready to  show that there exists a unique strong solution to the SDE (\ref{1.1}).
First of all, by \cite[Theorem 3]{Jin} and \cite[Theorem A]{ZhangQ} 
 we have the following result on the weak solution of SDE (\ref{1.1}):
\begin{thm} \label{T2.1}
Assume $|b| \in \K_{d,1}$. Then for each $x \in \Rd$, there exists a unique weak solution to the following SDE:
\begin{eqnarray}  \label{5.1} \nonumber
\begin{split}
X_t=x+W_t+\int_0^tb(s,X_s)\d s,\ \forall t>0.
\end{split}
\end{eqnarray}
Moreover, the solution admits a continuous transition density $p(t,x,s,y)$ satisfying that for any $T>0$, there exist positive constants $M_6$ and $M_7$ such that for any $0< s-t\le T$ and $x,y \in \Rd$,
\begin{equation}\label{2.1}
p(t,x,s,y)\leq M_6(s-t)^{- \frac{d}{2}}e^{-\frac{M_7|x-y|^2}{2(s-t)}}.
\end{equation}
\end{thm}

\vskip 0.3cm

To carry out Zvonkin transformation for SDEs (\ref{1.1}), we need the following  Krylov's type convergence.

\begin{lem}  \label{L3.6}
Assume $ |b| \in \K_{d,1}$. Let $X$ be a weak solution to the SDE (\ref{1.1}) wtth $X_0=x$. Then for any $h \in \K_{d,1}$ we have
\begin{eqnarray}  \nonumber
\begin{split}
& \lim_{n\to \infty}E[\sup_{0\le t \le T}|\int_0^th(s,X_s)\d s-\int_0^th_n(s,X_s)\d s|^2]=0,
\end{split}
\end{eqnarray}
where $h_n(t,x):=\int_{\R^{d+1}} \varphi_n(t-s,x-y)h(s, y)\d y \d s $.
\end{lem}
\begin{proof}
We first show that for $t>0$,
\begin{eqnarray}  \label{2.43}
\begin{split}
\lim_{n \to \infty}E[\int_0^th_n(s,X_s)\d s]=E[\int_0^th(s,X_s)\d s].
\end{split}
\end{eqnarray}

For any $\e>0$, by Lemma \ref{L3.2} and (\ref{2.1}), there exists $0<\delta<t$ such that
\begin{eqnarray}  \label{4.3-1}
\begin{split}
\sup_{n \ge 1} \int_0^{\delta} \int_{\Rd} p(0,x,s,y) (|h_n(s,y)|+|h (s,y)| )  \d y \d s \le \e.
\end{split}
\end{eqnarray}
Also by Lemma \ref{L2.2-1} and (\ref{2.1}) one can choose $R>0$ large enough so that
\begin{eqnarray}  \label{4.4-1}
\begin{split}
& \ \ \ \   \sup_{|\tau|\vee |z|< \frac{\delta}{2}} \int_{\delta - \tau}^{t- \tau} \int_{|x-y|\ge R}  (p(0,x,s+\tau,y+z) +p(0,x,s,y)) | h (s,y) | \d s \d y  \\
&\le c \delta^{-\frac{d}{2}} \sup_{x \in \Rd} \int_{\frac{\delta}{2}}^{t+\frac{\delta}{2} } \int_{|x-y|\ge R-\frac{\delta}{2}} e^{-\frac{M_7|x-y|^2}{4t}}  |h (s,y)|   \d y  \d s\\
& \le c \delta^{-\frac{d}{2}} e^{-\frac{M_7 (R-\frac{\delta}{2})^2}{8t}} \sup_{  x \in \Rd}  \int_{\frac{\delta}{2}}^{t+\frac{\delta}{2} } \int_{\Rd}e^{-\frac{M_7|x-y|^2}{8t}}  |h (s,y)|   \d y  \d s \le \e.
\end{split}
\end{eqnarray}
Take a positive integer $N$ such that $2^{-N}<\frac{\delta}{2}$. It follows from (\ref{4.3-1}), (\ref{4.4-1}) and (\ref{2.1}) that for $n \ge N$ and $R>0$,
\begin{eqnarray}  \nonumber
\begin{split}
&\ \ \ \ |E[\int_0^t(h_n(s,X_s)-h(s,X_s))\d s]|\\
&\le \sup_{n \ge 1} \int_0^{\delta} \int_{\Rd} p(0,x,s,y) (|h_n(s,y)|+|h (s,y)| ) \d s \d y \\
&\ \ \ \ +|\int_{\R^{d+1}} \varphi_n(\tau,z) \d \tau \d z  \int_{\delta - \tau}^{t- \tau} \int_{\Rd} (p(0,x,s+\tau,y+z)-p(0,x,s,y)) h (s,y)  \d s \d y |\\
&\ \ \ \ + \int_{\R^{d+1}} \varphi_n(\tau,z) \d \tau \d z  \int_{\delta - \tau}^{ \delta} \int_{\Rd}  p(0,x,s,y)  | h (s,y) | \d s \d y \\
&\ \ \ \ + \int_{\R^{d+1}} \varphi_n(\tau,z) \d \tau \d z  \int_{t - \tau }^{t } \int_{\Rd}  p(0,x,s,y)  | h (s,y) | \d s \d y \\
&\le \sup_{n \ge 1} \int_0^{\delta} \int_{\Rd} p(0,x,s,y) (|h_n(s,y)|+|h (s,y)| ) \d s \d y \\
&\ \ \ \ +  \sup_{|\tau|\vee |z|< \frac{\delta}{2}} \int_{\delta - \tau }^{t- \tau } \int_{|x-y|\ge R}  (p(0,x,s+\tau,y+z) +p(0,x,s,y)) | h (s,y) | \d s \d y   \\
&\ \ \ \ +|\int_{\R^{d+1}} \varphi_n(\tau,z) \d \tau \d z  \int_{\delta- \tau }^{t- \tau } \int_{B(x,R)} (p(0,x,s+\tau,y+z)-p(0,x,s,y)) h (s,y)  \d s \d y |\\
&\ \ \ \ +  \int_{0}^{ \delta} \int_{\Rd}  p(0,x,s,y)  | h (s,y) | \d s \d y +  \int_{t - 2^{-n} }^{t } \int_{\Rd}  p(0,x,s,y)  | h (s,y) | \d s \d y \\
&\le 3 \e +\sup_{ |\tau|\vee |z|< 2^{-n} , (s,y)\in (\frac{\delta}{2},t)\times B(x,R) }  |p(0,x,s+\tau,y+z)-p(0,x,s,y)|    \int_{0}^t\int_{B(x,R)}|h(s,y)| \d y \d s \\
 &\ \ \ \ +  \int_{t - 2^{-n} }^{t } \int_{\Rd}  p(0,x,s,y))  | h (s,y) | \d s \d y. \\
\end{split}
\end{eqnarray}
 Letting $n \to \infty$, by the continuity of $p(0,x,\cdot,\cdot)$ on $(0,\infty)\times \Rd$, we get
\begin{eqnarray}  \nonumber
\begin{split}
\lim_{n \to \infty} |E[\int_0^t(h_n(s,X_s)-h(s,X_s))\d s]| \le 3 \e  .
\end{split}
\end{eqnarray}
Since $\e$ is arbitrary, we obtain (\ref{2.43}).

\vskip 0.3cm

 Now let  $v_n(t,x):=\int_t^T\int_{\Rd} q(s-t,x,y) h_n(s,y) \d y \d s$ and $v(t,x):=\int_t^T\int_{\Rd} q(s-t,x,y) h(s,y) \d y \d s$. Then $v_n$ and $v$ are the solutions of the following equations:
\begin{align}  \nonumber
\left\{
\begin{aligned}
&\partial_t v_n (t,x)+\frac{1}{2}\Delta v_n(t,x) =-h_n(t,x),&&\forall (t,x) \in (0,T)\times \Rd, \\
&  v_n(T,x) = 0, && \forall x \in \Rd.
\end{aligned}
\right.
\end{align}
\begin{align}  \nonumber
\left\{
\begin{aligned}
&\partial_t v (t,x)+\frac{1}{2}\Delta v(t,x)=-h(t,x),&&\forall (t,x) \in (0,T)\times \Rd, \\
& v(T,x) = 0, && \forall x \in \Rd.
\end{aligned}
\right.
\end{align}
Applying Ito's formula we obtain
\begin{eqnarray}  \label{2.45}
\begin{split}
v_n (t,X_t)=&v_n (0,x)+\int_0^t \<\nabla_xv_n (s,X_s), \d W_s \>\\
&+\int_0^t \<\nabla_xv_n (s,X_s),  b(s,X_s)\> \d s-\int_0^t h_n(s,X_s) \d s.
\end{split}
\end{eqnarray}
Combining Proposition \ref{L3.3} and (\ref{2.43}), letting $n \to \infty$ in (\ref{2.45}) we get
\begin{eqnarray}  \label{2.46}
\begin{split}
v  (t,X_t)=&v  (0,x)+\int_0^t \<\nabla_xv  (s,X_s), \d W_s \>\\
&+\int_0^t \<\nabla_xv  (s,X_s),  b(s,X_s)\> \d s-\int_0^t h (s,X_s) \d s.
\end{split}
\end{eqnarray}
Using the Gaussian upper bound estimate of the density function of $X$ and $|b| \in \K_{d,1}$, one can easily see that
\begin{eqnarray}  \label{5.7}
\begin{split}
E[(\int_0^T|b(s,X_s)| \d s)^2]< \infty. \\
\end{split}
\end{eqnarray}
Set $\tau_R:=\inf\{t \ge 0: |X_t|\ge R\}$ for $R>0$. Then by Proposition \ref{L3.3} , (\ref{2.45})-(\ref{5.7}) and Burkholder's inequality we have
\begin{eqnarray}  \nonumber
\begin{split}
&\ \ \ \  \lim_{n\to \infty}E[\sup_{0\le t \le T}|\int_0^th(s,X_s)\d s-\int_0^th_n(s,X_s)\d s|^2]\\
&\le 2  \lim_{n\to \infty}E[\sup_{0\le t \le T}|v(t,X_t)- v_n(t,X_t)|^2|]+4\lim_{n\to \infty}E[  \int_0^T |\nabla_xv(s,X_s) - \nabla_xv_n(s,X_s)|^2\d s]\\
&\ \ \ \ + \lim_{n\to \infty}E[(\int_0^t |\nabla_x v  (s,X_s)- \nabla_x v_n  (s,X_s)| |b(s,X_s)| \d s)^2] \\
&\le c    (P[ T>\tau_R] + \lim_{n \to \infty}     \|v -v_n \|_{C_b^{0,1}( [0,T] \times B(0,R))})=c P[ T>\tau_R].
\end{split}
\end{eqnarray}
Let $R \to \infty$ to obtain
\begin{eqnarray}  \nonumber
\begin{split}
\lim_{n\to \infty}E[\sup_{0\le t \le T}|\int_0^th(s,X_s)\d s-\int_0^th_n(s,X_s)\d s|^2]=0.
\end{split}
\end{eqnarray}
\end{proof}

\vskip 0.4cm

Next result is the  Zvonkin transform of SDEs (\ref{1.1}).

\begin{prp}  \label{P3.7}
Assume $|b|, |f| \in \K_{d,1}$. Let $X$ be a solution to the SDE (\ref{1.1}) and $u$ the mild solution to the equation (\ref{2.13}). Then we have for any $0 <t \le T$,
\begin{eqnarray}  \label{2.44}
\begin{split}
u (t,X_t)=&u (0,x)+\int_0^t \<\nabla_xu (s,X_s), \d W_s \>+\int_0^t f (s,X_s) \d s.
\end{split}
\end{eqnarray}
\end{prp}
\begin{proof}
Let $u_n$ be the classical solution to equation (\ref{2.17}). Then by the Ito's formula,
\begin{eqnarray}  \label{2.47}
\begin{split}
u_n(t,X_t)=&u_n(0,x)+\int_0^t \<\nabla_xu_n(s,X_s), \d W_s \>\\
&+\int_0^t \<\nabla_xu_n(s,X_s), b(s,X_s)-b_n(s,X_s)\> \d s+\int_0^t f_n(s,X_s) \d s.
\end{split}
\end{eqnarray}
By Lemma \ref{L3.6}, there exists a subsequence $\{n_k\}_{k \ge 1}$ such that for $P$-$a.e. \omega \in \Omega$ and $t \in [0,T]$,
\begin{eqnarray}  \nonumber
\begin{split}
\lim_{k \to \infty}\int_0^t   b_{n_k}(s,X_s(\omega))  \d s=\int_0^t b(s,X_s(\omega)) \d s.
\end{split}
\end{eqnarray}
Since $\nabla_xu (s,X_s)$ is continuous with respect to $s$, it follows that
\begin{eqnarray}  \nonumber
\begin{split}
\lim_{k \to \infty}\int_0^t \<\nabla_xu (s,X_s), b_{n_k}(s,X_s)-b(s,X_s)\> \d s=0.
\end{split}
\end{eqnarray}
Combining this with Proposition \ref{L3.3} and Lemma \ref{L3.6}, letting $n \to \infty$ in (\ref{2.47}) we get (\ref{2.44}).
\end{proof}

\vskip 0.4cm

Now we are ready to state the main result:

\begin{thm}  \label{T4.3}
Assume for any $T>0$, $|b(t,x)|^2 I_{(0,T)}(t) \in \K_{d,\a}$ for some $\a \in (0,2)$. Then for each $x \in \Rd$, there exists a unique strong solution to SDE (\ref{1.1}).
\end{thm}

\begin{proof}
 By the Yamada-Watanabe theorem and Theorem \ref{T2.1}, it suffices  to prove the pathwise uniqueness of the solution to SDE (\ref{1.1}).
Assume $X$ and $Y$ are solutions to the SDE (\ref{1.1}) with $X_0=Y_0=x$. For any given $T_0>0$, set $b_{T_0}(t,x):=b(t,x) I_{t<T_0}$. Then by Corollary \ref{C2.3} and Lemma \ref{L2.4}, $|b_{T_0}|^2$ satisfies Hypothesis $\mathrm{H_p}$ and $|b_{T_0}| \in \K_{d,1}$. Let $u$ be the  unique mild solution to the following parabolic equation:
\begin{align}  \label{2.11}
\left\{
\begin{aligned}
&\partial_t   u (t,x)+\frac{1}{2}\Delta   u(t,x)+\< b_{T_0}(t,x),  \nabla_x   u(t,x) \>   =b_{T_0},&&\forall (t,x) \in (0,T)\times \Rd, \\
&   u(T,x) = 0, && \forall x \in \Rd.
\end{aligned} \nonumber
\right.
\end{align}
Letting $v(t,x):=x-u(t,x)$, by Proposition \ref{L3.3},   there exist $0<c_1<c_2$ and constant $T>0$ such that for $t \in [0,T]$ and $x,y \in \Rd$,
\begin{eqnarray}  \label{2.48}
\begin{split}
c_1|x-y| \le |v(t,x)-v(t,y)| \le c_2|x-y|.
\end{split}
\end{eqnarray}
Applying Proposition \ref{P3.7}, we obtain
\begin{eqnarray}  \nonumber
\begin{split}
v (t,X_t)-v(t,Y_t)= \int_0^t \<\nabla_xu (s,X_s)-\nabla_xu (s,Y_s), \d W_s \>.
\end{split}
\end{eqnarray}
Let $\tau_R:=\inf \{t\le T\wedge T_0: |X_t-x|\vee |Y_t-x|\ge \frac{R}{2}\}$. Then by Lemma \ref{L2.6} and (\ref{2.48}), for any $t \in [0, \tau_R ]$,
\begin{eqnarray}  \label{2.48-1}
\begin{split}
&\ \ \ \  |X_t -Y_t|^2 \\
&\le c  |v (t,X_t)-v(t,Y_t)|^2 \\
&\le c \int_0^t (v (s,X_s)-v(s,Y_s))\<\nabla_xu (s,X_s)-\nabla_xu (s,Y_s), \d W_s \>\\
&\ \ \ \ + c \int_0^{ t} \|\nabla_xu (s,X_s)-\nabla_xu (s,Y_s) \|^2 \d  s \\
&\le c \int_0^t (v (s,X_s)-v(s,Y_s))\<\nabla_xu (s,X_s)-\nabla_xu (s,Y_s), \d W_s \>\\
&\ \ \ \ +c\int_0^{ t} |X_s-Y_s|^2(|M_R \|D^2_x u\| (s,X_s)|^2+|M_R \|D^2_x u\| (s,Y_s) |^2) \d  s .
\end{split}
\end{eqnarray}
By Lemma \ref{L2.5} and Proposition \ref{P3.5}, we know that $|M
_R\|D^2_x u\||^2$ satisfies the Hypothesis  $\mathrm{H_q}$ for some $q<p$. In particular,  $|M_R\|D^2_x u\||^2 \in \K_{d,2}$ by Corollary \ref{C2.3}. Using the upper bound of the density function $p(t,x,s,y)$ in (\ref{2.1}), we see that for any $t>0$,
\begin{eqnarray}  \nonumber
\begin{split}
E[\int_0^{t}  (|M_R \|D^2_x u\| (s,X_s)|^2+|M_R \|D^2_x u\| (s,Y_s) |^2) \d  s  ]<\infty.
\end{split}
\end{eqnarray}
Applying the stochastic Gronwall inequality (see $e.g.$ \cite[Theorem 4]{Scheutzow}), we deduce from (\ref{2.48-1}) that  for any $R>0$,
\begin{eqnarray}  \nonumber
\begin{split}
E[\sup_{t \le \tau_R } |X_t-Y_t|^2]=0.
\end{split}
\end{eqnarray}
Letting $R\rightarrow \infty$, we see that
\begin{eqnarray}  \nonumber
\begin{split}
E[\sup_{t \le T \wedge T_0} |X_t-Y_t|^2]=0.
\end{split}
\end{eqnarray}
Since the constant $T$ is independent of the initial value  $x$, using standard arguments, we conclude  that $X_t=Y_t$ for all $t \in [0,T_0]$ $P$-$a.e.$. Since $T_0$ is arbitrary, we have proven  the pathwise uniqueness.
\end{proof}

\medskip

\noindent{\bf  Acknowledgement.}\   This work is partly
supported by National Natural Science Foundation of China (NSFC) (No. 11971456, No. 11671372, No. 11721101).

\end{document}